\pdfoutput=1

\documentclass{amsart}

\usepackage{verbatim}
\usepackage{xcolor}
\usepackage{tikz}
\usepackage{tikz-cd}
\usetikzlibrary{arrows.meta}
\usepackage{arydshln}
\usepackage{caption}
\usepackage{subcaption}
\usepackage{graphicx} 
\usepackage[hyphens]{url}
\usepackage{amsmath}
\usepackage{bbm}
\usepackage{hyperref}
\usepackage{mathrsfs}
\usepackage{todonotes}
\usepackage{amssymb}
\usepackage{mathabx}
\usepackage{enumitem}

\newtheorem{theorem}{Theorem}[section]
\newtheorem{proposition}[theorem]{Proposition}
\newtheorem{lemma}[theorem]{Lemma}
\newtheorem{corollary}[theorem]{Corollary}

\newtheorem{sublemma}[theorem]{Sublemma}
\newtheorem{meta}[theorem]{Metaconjecture}

\theoremstyle{definition}
\newtheorem{definition}[theorem]{Definition}
\newtheorem{example}[theorem]{Example}

\theoremstyle{remark}
\newtheorem{remark}[theorem]{Remark}

\newtheorem{question}[theorem]{Question}

\numberwithin{equation}{section}

\newcommand{\cX}{\mathcal{X}}
\newcommand{\cU}{\mathcal{U}}

\newcommand{\bR}{\mathbb{R}}
\newcommand{\mfS}{\mathfrak{S}}

\newcommand{\Tw}{\mathrm{Tw}}
\newcommand{\Stab}{\mathrm{Stab}}

\newcommand{\e}{\varepsilon}

\renewcommand{\Im}{\mathrm{Im}}

\newcommand{\MCG}{\mathrm{MCG}}
\newcommand{\Cay}{\mathrm{Cay}}
\newcommand{\twist}{\mathrm{twist}}
\newcommand{\modulus}{\mathrm{mod}}

\newcommand{\cI}{\mathcal{I}}
\newcommand{\cE}{\mathcal{E}}
\newcommand{\cH}{\mathcal{H}}

\newcommand{\bC}{\mathbb{C}}
\newcommand{\cM}{\mathcal{M}}

\newcommand{\bZ}{\mathbb{Z}}
\newcommand{\cC}{\mathcal{C}}
\newcommand{\cD}{\mathcal{D}}
\newcommand{\cT}{\mathcal{T}}

\newcommand{\cS}{\mathcal{S}}

\newcommand{\bH}{\mathbb{H}}
\newcommand{\bP}{\mathbb{P}}

\newcommand{\ol}[1]{\makebox[0pt]{$\phantom{#1}\overline{\phantom{#1}}$}#1}

\newcommand{\approxadd}{\overset{+}{\asymp}}
\newcommand{\approxmult}{\overset{*}{\asymp}}

\renewcommand{\bold}[1]{\medskip \noindent {\bf #1 }\nopagebreak}

\begin{document}

\title[The geometry of totally geodesic subvarieties]{The geometry of totally geodesic subvarieties of moduli spaces of Riemann surfaces}

\author{Francisco Arana--Herrera}

\author{Alex Wright}

\begin{abstract}
We prove a semisimplicity result for the  boundary, in the corresponding Deligne--Mumford compactification, of a totally geodesic subvariety of a moduli space of Riemann surfaces. At the level of Teichm\"uller space, this semisimplicity theorem gives that each component of the boundary is a product of simple factors, each of which behaves metrically like a diagonal embedding. Building on this result, we also show that the associated totally geodesic submanifolds of Teichmüller space and orbifold fundamental groups are hierarchically hyperbolic. 

The proof intertwines in a novel way results and perspectives originating in dynamics, algebraic geometry, geometric group theory, and both classical and modern Teichm\"uller theory. It establishes both new rigidity and new flexibility for totally geodesic submanifolds and their associated varieties and orbifold fundamental groups  and provides a rich set of new  tools for the study of these objects.  
\end{abstract}

\maketitle

\thispagestyle{empty}

\tableofcontents

\newpage 

\section{Introduction}

\subsection{Initial context.} Let $\cT_{g,n}$ be the Teichm\"uller space of genus $g$ Riemann surfaces with $n$ punctures or marked points,  and let $$\pi:\cT_{g,n} \to \cM_{g,n}$$ be the map to the associated moduli space of Riemann surfaces $\cM_{g,n}$. We also use the same notation for the associated maps on the bundles of quadratic differentials. 

Teichm\"uller space admits several metrics of interest, but here we consider exclusively the Teichm\"uller metric. This metric reflects the modular nature of Teichm\"uller space, giving a concrete measure of how non-conformal a map between two Riemann surfaces must be, and also the intrinsic complex geometry of Teichm\"uller space as a complex manifold, being equal to the Koyabashi metric \cite{Royden}. It is in a sense as inhomogeneous as possible \cite{Royden} and its fine geometry is often mysterious. 

Teichm\"uller discs, also known as complex geodesics, are holomorphic isometric embeddings of the hyperbolic plane into Teichm\"uller space. There is a unique Teichm\"uller disc through any pair of distinct points, and one might say that the mystery of Teichm\"uller and moduli spaces can sometimes be pierced  with Teichm\"uller discs. 

Here we study totally geodesic submanifolds of Teichm\"uller space. We say a complex submanifold $N$ of Teichm\"uller space is totally geodesic if $N$ contains the Teichm\"uller disc\footnote{ Equivalently, the (real) Teichm\"uller geodesic.
} through every pair of distinct points in $N$.

The one complex dimensional totally geodesic submanifolds are exactly the Teichm\"uller discs.  We call a complex manifold or variety higher dimensional if it has dimension greater than one. 

Starting with smaller Teichm\"uller spaces and taking covering constructions in the sense of \cite[Section 6]{MMW}, one can obtain higher dimensional totally geodesic submanifolds of larger Teichm\"uller spaces, which we call trivial. Recently, non-trivial, higher dimensional examples were discovered for the first time \cite{MMW, EMMW}. 

A totally geodesic submanifold $N\subseteq \cT_{g,n}$ is called algebraic if $\pi(N) \subseteq \mathcal{M}_{g,n}$ is a closed algebraic variety; in this case $\pi(N)$ is called a totally geodesic subvariety. The non-trivial examples mentioned above are all algebraic, and were discovered via their associated totally geodesic subvarieties. Later, the second author established the following \cite{Wri20}.

\begin{theorem}[Wright]
Every higher dimensional totally geodesic submanifold is algebraic. There are only finitely many higher dimensional totally geodesic subvarieties in each moduli space.
\end{theorem}

For a typical Teichm\"uller disc $N \subseteq \mathcal{T}_{g,n}$, its projection $\pi(N)$ is dense in $\cM_{g,n}$, and each moduli space $\cM_{g,n}$ of complex dimension greater than one has infinitely many totally geodesic subvarieties of dimension one, typically known as Teichm\"uller curves. So the previous theorem establishes a very strong contrast between the cases of dimension one and higher dimension. 

Already in \cite{MMW} it was suggested that the example constructed there might be seen as ``a new type of Teichm\"uller space, on an equal footing with $\cT_{g,n}$''. Indeed, we believe that that higher dimensional totally geodesic submanifolds are as complicated as, and behave very much like, whole Teichm\"uller spaces. We elaborate on this as follows: 

\begin{meta}\label{M}
Many results known for Teichm\"uller spaces also hold for higher dimensional totally geodesic submanifolds;  many results known for moduli spaces also hold for higher dimensional totally geodesic subvarieties; and many results known for mapping class groups also hold for the orbifold fundamental groups of higher dimensional totally geodesic subvarieties. 
\end{meta}

Theorem \ref{T:main2} below can be seen as an instance of this metaconjecture. 

\subsection{New results.} In this paper we make use of augmented Teichm\"uller space $\ol{\cT}_{g,n}$, which is obtained from Teichm\"uller space by adding nodal surfaces where collections of simple closed curves have been pinched. A more descriptive but less common name for  augmented Teichm\"uller space is the Deligne--Mumford bordification of Teichm\"uller space, and indeed the quotient of $\ol{\cT}_{g,n}$ by the corresponding mapping class group is the Deligne--Mumford compactification $\ol{\cM}_{g,n}$ of moduli space. See \cite[Section 5.2]{BainbridgeThesis} for an expository introduction. 

The Deligne--Mumford bordification is stratified and  each stratum is a product of Teichm\"uller spaces.  A complex submanifold $L$ of a product of Teichm\"uller spaces is called simple if its projection to each factor is an isometric embedding whose image is a totally geodesic submanifold; see the next paragraph for details. A complex submanifold $L$ of a product of Teichm\"uller spaces is called semisimple if it is a product of simple factors. 

In the definition of \textit{simple} above we use the sup metric on the product of Teichm\"uller spaces. The isometric embedding condition can be rephrased without referencing a metric on the product by saying that given two points $(X_1, \ldots, X_k)$ and $(Y_1, \ldots, Y_k)$ in $L$ we have $d(X_i, Y_i) = d(Y_j, X_j)$ for all $i,j \in \{1,\dots,k\}$. In this sense, the isometric embedding condition indicates that $L$ looks metrically like a diagonally embedded copy of a subset of $\cT_{g',n'}$ in $(\cT_{g',n'})^k$. 

Our first main result is the following, where $\ol{N} \subseteq \ol{\cT}_{g,n}$ denotes the closure of $N \subseteq \mathcal{T}_{g,n}$ in the Deligne--Mumford bordification. 

\begin{theorem}\label{T:main1}
Let $N \subseteq \mathcal{T}_{g,n}$ be an algebraic totally geodesic submanifold of Teichm\"uller space and let $L$ be the intersection of $\ol{N}$ with a stratum of $\ol{\cT}_{g,n}$. Then $L$ is semisimple and algebraic. 
\end{theorem}

The main point of Theorem \ref{T:main1} is that $L$ is semisimple; the algebraicity is not hard and is included for completeness. Implicit in Theorem \ref{T:main1} is that $L$ is a connected complex manifold, or in other words it is irreducible and smooth.

\begin{example}\label{E:FirstExample}
There is a trivial totally geodesic submanifold $N\subseteq \cT_{3}$ that can be defined as the fixed point set of an involution in the mapping class group, and which  consists of unbranched double covers of genus two surfaces.  Consider such a cover $X\to Y$ and let $\alpha$ be a separating simple closed curve on $Y$ whose preimage on $X$ has two components; see Figure \ref{F:FirstExample}. Pinching this preimage multi-curve gives rise to a component $$L \subseteq \cT_{1,1} \times \cT_{1,1} \times \cT_{1,2} $$ of  $\ol{N}$. If $L_1\subseteq \cT_{1,1} \times \cT_{1,1}$ is the diagonal and $L_2\subseteq \cT_{1,2}$ is the set of $(Z, \{p,q\})$ where the two marked points $\{p,q\}$ are exchanged by an appropriate involution of $Z$, we have $L=L_1\times L_2$. Since $L_1$ and  $L_2$ are each simple,  $L$ is semisimple. 
\end{example}

\begin{figure}[h!]
\includegraphics[width=0.6\linewidth]{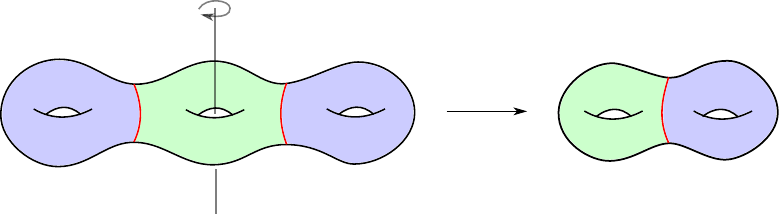}
\caption{The cover $X\to Y$ in Example \ref{E:FirstExample}.}
\label{F:FirstExample}
\end{figure}  

Our second main result is the following, whose proof uses Theorem \ref{T:main1}. 

\begin{theorem}\label{T:main2}
Let $N \subseteq \mathcal{T}_{g,n}$ be an algebraic totally geodesic submanifold of Teichm\"uller space. Then $N$ is a hierarchically hyperbolic space and the associated orbifold fundamental group  is a hierarchically hyperbolic group.
\end{theorem}

By ``the associated orbifold fundamental group" we mean the stabilizer of $N$ in the mapping class group, or equivalently the orbifold fundamental group of the normalization of $\pi(N)\subseteq \cM_{g,n}.$ 

Roughly speaking, Theorem \ref{T:main2} indicates that we can use curve graphs of subsurfaces as a coarse coordinate system, allowing us to study these objects combinatorially. To do this we only use subsurfaces that appear as nodal subsurfaces in $\ol{N}$. We define two subsurfaces to be equivalent if they occur in the same simple factor, and, roughly speaking, we use one subsurface from each equivalence class. In $N$ the shape  of an active subsurface roughly determines that of all equivalent subsurfaces, and, in contrast, the shapes of appropriate nonequivalent subsurfaces are independent. This gives both new rigidity phenomenon (equivalent subsurfaces) and new flexibility (independent subsurfaces). 

Theorem \ref{T:main2} also encodes a large amount of information and structure on which strata of the Deligne-Mumford bordification $\ol{N}$ intersects. 

\bold{Independent work.} Benirschke, Dozier, and Rached \cite{BDR} have independently and simultaneously obtained a version of Theorem \ref{T:main1} (but without smoothness or irreducibility). Their analysis involves significant use of Beltrami differentials, which do not appear in our work, and conversely most of our main tools do not appear in their work. Their is however some overlap, in particular at a key step of our analysis in Section \ref{SS:Primality}. 

\bold{Overview of motivation.} Theorem \ref{T:main1} implies that totally geodesic subvarieties have remarkably understandable closures in the corresponding Deligne--Mumford compactification. This could prove to be very useful for the dynamical and algebro-geometric study of these subvarieties.  

Theorem \ref{T:main2} implies that a number of previous results for Teichmüller spaces, some mentioned below, now apply to clarify the coarse geometry of totally geodesic submanifolds. 

Theorem \ref{T:main2} also  establishes the objects in question as potentially clarifying examples in the theory of subgroups of mapping class groups and in the theory of hierarchical hyperbolicity. For instance, the fundamental groups of higher dimensional totally geodesic subvarieties may serve as prototypical examples  in the search for a useful general definition of geometrically finite subgroups of mapping class groups, beyond the special case of parabolically geometrically finite subgroups, and some of the ideas behind Theorem \ref{T:main2} may provide a prototype for understanding certain  nice subsets of hierarchically hyperbolic spaces which are not hierarchically quasi-convex. See \cite[Definition 1.10]{ExtensionsII} and \cite{Convexity} for context. 

There is also a significant potential for our results to provide  new tools in the study of the classification problem of higher dimensional totally geodesic subvarieties. We will comment on this in more detail below, after discussing the relationship to invariant subvarieties of differentials and saying more about the known examples. 

\bold{Invariant subvarieties of differentials.} There is a $GL(2,\mathbb{R})$ action on the bundle of quadratic differentials (non-zero, holomorphic away from marked points, and with at most simple poles at marked points) over $\cM_{g,n}$, and deep results imply that its orbit closures are complex algebraic varieties \cite{EMM15, EMbig, Filip}. Let $\rho$
be the natural forgetful map defined by $\rho(X,q) := X$.
As we recall in Remark \ref{R:2dim} below, if $\cM$ is a closed invariant subvariety of quadratic differentials, then $\rho(\cM)$ is a totally geodesic subvariety if $$\dim(\cM) = 2 \dim \rho(\cM),$$ and all totally geodesic subvarieties arise in this way. 

In general there seems to be no reason to expect $\dim(\cM) = 2 \dim \rho(\cM)$ when $\dim(\cM)>2$. All that is known at present is that if $\cM$ intersects the principle stratum of quadratic differentials this condition is automatically satisfied \cite[Proposition 5.1]{HT}, and if $\cM$ is not transverse to the isoperiodic foliation this condition is never satisfied \cite[Theorem 1.3]{Wri20}.

\bold{Recent examples.} Recently, a collection of six new invariant subvarieties of quadratic differentials were discovered \cite{MMW, EMMW}, all of dimension four. Of these, three satisfy the condition $\dim(\cM) = 2 \dim \rho(\cM)$ and hence give rise to totally geodesic subvarieties in moduli space. These are complex surfaces in $\cM_{1,3}$, $\cM_{1,4}$, and $\cM_{2,1}$.

\bold{Hopes for classification.} It is a mystery how many higher dimensional totally geodesic subvarieties of moduli space remain to be discovered. Given how few are known at present, one might hope for a complete classification.

One approach to  this problem is as follows: First, classify all invariant subvarieties of quadratic differentials, then, determine which give rise to totally geodesic subvarieties of moduli space. The first step is a major goal in its own right and has seen major progress in recent years; see \cite{McMullenGenus2, Genus3, HypApisa, Hull, HighRank, Ygouf, Saturated} and the surveys \cite{McMullenSurvey, Broad} for a very incomplete selection of results. That being said, the problem of finding  a complete classification of all invariant subvarieties of quadratic differentials remains extremely challenging and seems to us unlikely to be completed anytime soon. 

This raises the question of if it is possible to approach the classification of totally geodesic subvarieties of moduli space in a new way, making more use of the totally geodesic assumption and not merely the general features of the associated invariant subvariety of quadratic differentials. We believe that the answer to this question is yes, and the results we present in this paper provide new tools for this purpose.

In particular,  Theorem \ref{T:main1} allows for inductive arguments, but this may well be only the beginning. One might dream that Theorem \ref{T:main1} could lead to new insights related to the fact that the Teichm\"uller metric stores an enormous amount of useful information. Recall for example that complex linear isometric embeddings $Q(Y) \to Q(X)$ of spaces of quadratic differentials must arise from maps $X\to Y$ of the Riemann surfaces \cite[Theorem 1.3]{GG} and that isometric embeddings of Teichm\"uller spaces must arise from covering constructions \cite{IsomEmb}.

Putting aside such uncertain dreams, we wish to give a deeper explanation of how our results seem likely to lead to major improvements of existing techniques for the classification problem. Let us start with the Cylinder Deformation Theorem of the second author \cite{Wcyl}, which has been generalized in \cite{Equations}, and is now a primary tool in the classification problem. Given an invariant subvariety $\cM$, and a differential $(X,q)\in \cM$, one defines an equivalence relation on cylinders of $q$, with two cylinders being equivalent if they are parallel and remain parallel on nearby differentials in $\cM$. The Cylinder Deformation Theorem states that equivalence classes of cylinders can be deformed while remaining in  $\cM$. The rigidity provided by the equivalence relation is in tension to the flexibility provided by the deformations, giving the first indication of the potential usefulness of this theorem. 

One of our intermediate results allows us to consider simple closed curves on surfaces rather than cylinders on quadratic differentials, raising the possibility of easier and more topological analogues of arguments involving the Cylinder Deformation Theorem. Let $N \subseteq \mathcal{T}_{g,n}$ be an algebraic totally geodesic submanifold. We define $QN$ to be the set of quadratic differentials generating Teichm\"uller discs completely contained in $N$. Note that $\pi(QN)$ is an algebraic variety \cite{EMM15, EMbig, Filip}. It is not hard to show $QN$ is irreducible, and \cite[Proposition 3.1]{IsomEmb} shows the harder fact that it is a holomorphic vector bundle over $N$. We define $Q_{\max}N$ to be the largest stratum of $QN$; see Section \ref{s:cyl} for details. We say that a simple closed curve is a cylinder at a quadratic differential if it is homotopic to the core curve of a cylinder on that differential.
\begin{theorem}\label{T:SummaryForClassification}
Suppose $N\subseteq \cT_{g,n}$ is an algebraic totally geodesic submanifold. Let $\cS$ be the set of simple closed curves on the marking surface that are cylinders at some point of $Q_{\max}N$ and define two such curves to be equivalent if there is a point of $Q_{\max}N$ where they are equivalent cylinders. 
\begin{enumerate}
\item\label{SFC:CylinderEverywhere} For every $X\in N$ and every $\alpha\in \cS$  there is a quadratic differential $(X,q)\in QN$ where $\alpha$ is a cylinder. 
\item\label{SFC:EquivAlsoCylinder} If $\alpha\in \cS$ is a cylinder at $(X,q)\in QN$, then the equivalent curves in $\cS$ are also cylinders at $(X,q)$ and these cylinders are equivalent. The ratios of  circumferences and the ratios of moduli of these cylinders depends only on their core curves and not on $(X,q)$. 
\item\label{SFC:SimultaneousRealization}  For any two simple closed curves $\alpha, \beta\in \cS$, there is a quadratic differential in $QN$ where they are both cylinders. 
\item\label{SFC:Pinchable}  Unions of disjoint equivalence classes in $\cS$ are exactly the multi-curves  that can be pinched in $\ol{N}$. 
\item\label{SFC:BoundaryOfFilled}  For any two equivalence classes of simple closed curves in $\cS$,  the boundary of the subsurface they fill is a union of equivalence classes of $\cS$.
\end{enumerate} 
\end{theorem}

In addition to this result, our proof of Theorem \ref{T:main2} extends the definition of equivalent simple closed curves to a definition of equivalent subsurfaces. The structure consisting of all these curves and subsurfaces and the equivalence relation on them is vaguely reminiscent of the structure of a root system and may be useful for the classification problem.  

\bold{Additional context on totally geodesic submanifolds.}  See \cite{GoujardSurvey} for a recent survey on totally geodesic submanifolds written to accompany a Bourbaki seminar on the subject. 

The three special examples discussed above are remarkable from many points of view, sharing connections with billiards in quadrilaterals and classical algebraic geometry, and have already been the topic of further study \cite{GothicVol, GothicIntersection}.

Totally geodesic submanifolds of Teichmüller space which are  isometric to covering constructions are known to arise from covering constructions \cite{IsomEmb}. 

There are no higher dimensional  analogues of Teichm\"uller discs which are symmetric spaces \cite{Ant2}. The holomorphicity in the definition of Teichm\"uller discs is essentially redundant \cite{Ant1}. See also \cite{Ant3}.

\bold{The core of Theorem \ref{T:main2}.} A series of papers beginning with seminal work of Masur and Minsky developed the idea that curve graphs can be used as coarse coordinate systems for mapping class groups and Teichm\"uller spaces; see \cite{MMI, MMII, OriginalRealization, RafiComb} for some of the key entries in this story. In particular, there is a map from Teichm\"uller space to the product of the curve graphs of all subsurfaces whose image is characterized by remarkably simple consistency inequalities, each involving only two subsurfaces at a time; the fact that the consistency inequalities characterize the image is called the Consistency Theorem  \cite{OriginalRealization} or the Realization Theorem \cite{HHS2}. The so-called Distance Formula allows one to coarsely compute distances in Teichm\"uller space using the image of this map. 

Roughly speaking, Theorem \ref{T:main2} indicates that we can pick out a collection of subsurfaces that are distinguished enough so that the image of $N$ in the corresponding product of curve graphs is characterized by similar consistency inequalities, and large enough so that the Distance Formula still holds. It would be easy to maintain the Distance Formula if we used all subsurfaces, but this would cause the Realization Theorem to fail. Conversely, it would be easy to understand the image if we used only a very small collection of subsurfaces, but using too small of a collection would cause the Distance Formula formula to fail. To prove Theorem \ref{T:main2} we must use exactly the right collection, allowing us to coarsely specify points in $N$ using their images in curve graphs uniquely and without unnecessary redundant information.   

\bold{Additional context on hierarchical hyperbolicity.} The relevant structure required to have a version of the Realization Theorem and the Distance Formula was axiomatized in influential work of Behrstock, Hagen, and Sisto \cite{HHS1, HHS2}. We recommend the survey \cite{WhatIs} for an introduction. The resulting structure, called a hierarchically hyperbolic structure, is not generic, but is enormously powerful: Even when given a space with nice non-positively curved behavior, one typically does not expect it to be hierarchically hyperbolic, but if it is hierarchically hyperbolic then one gains a remarkably good understanding of the space. Despite the stringent requirements, many spaces and groups have been shown to be hierarchically hyperbolic  \cite{HHS1, HHS2, RefinedCobination, Chesser, ExtraLarge, Equivariant3Groups, GraphProducts2, Vokes, ExtensionsII, MultiCurveExtensions, Hughes, admissible, CylinderGraph, Kopreski}. 

Other perspectives on hierarchically hyperbolicity exist beyond the coordinate system point of view we highlighted above. For instance, while hyperbolic spaces satisfy that hulls of finite subsets are coarsely trees, hierarchically hyperbolic spaces satisfy that hulls of finite subsets are coarsely $\mathrm{CAT}(0)$ cube complexes \cite{Quasiflats, StableCubulations, AsymptoticD}. Furthermore, hierarchically hyperbolic spaces are hyperbolic up to products of simpler hierarchically hyperbolic spaces \cite{AsymptoticD}. 

Consequences of hierarchical hyperbolicity, sometimes with extra restrictions satisfied in our examples, include: bounds on asymptotic dimension \cite{AsymptoticD},  classification of maximal dimensional quasi-flats \cite{Quasiflats}, coarse injectivity \cite{Helly},  semi-hyperbolicity \cite{Helly, StableCubulations}, and being quasi-cubical \cite{Petyt}. 

The specific way in which Theorem \ref{T:main2} holds has consequences like the fundamental group being undistorted in the mapping class group, which in the case of Teichm\"uller curves was proven in \cite{AffineUndistorted} and in the case of examples arising from covering constructions corresponds to \cite[Theorem 9.1]{RScovers}.

\bold{Some comments on the proofs.} 
Our analysis relies on and strengthens a profitable marriage between coarse geometry and results of a more analytic, algebro-geometric, and dynamical flavor. In particular, we make use of theorems of Rafi expressing how the coarse geometry of the Teichm\"uller metric affects the fine geometry of Teichm\"uller geodesics. Even our proof of Theorem \ref{T:main1}, a statement in fine, non-coarse geometry, is intimately linked with coarse geometry.

Our inspiration for why Theorem \ref{T:main1} should be true was, in the first place, the  structure theorem \cite[Theorem 1.3]{ChenWright} on the boundary of invariant subvarieties of Abelian differentials, and, secondly, the rather natural idea that Theorem \ref{T:main2} should not only be true at a coarse level but also be witnessed by fine geometry. 

Having proved Theorem \ref{T:main1}, the main difficulty in proving Theorem \ref{T:main2} is obtaining a coarse statement analogous to Theorem \ref{T:main1} but for curve graphs. This is surprisingly non-trivial to obtain from first principles but is well within the reach of existing technology. It requires the proof of a result of independent interest that does not seem to have been previously recorded in the literature, namely that the electrification of the real line parameterizing a Teichm\"uller geodesic along the intervals where curves are short quasi-isometrically embeds into the curve graph of the surface. 

\bold{Conventions.} We emphasize that we work only with complex submanifolds.
All dimensions we give are over $\bC$. 

When we state a result related to the Teichm\"uller space $\cT_{g,n}$, we allow all constants to implicitly depend on $g$ and $n$ without further comment; so for the purposes of such constants one should consider the Teichm\"uller space $\cT_{g,n}$ as fixed. 

\bold{Acknowledgments.} We thank 
Janusz Adamus,
Matt Bainbridge, 
Dawei Chen, 
Howard Masur, 
Jesús Ruiz,  
Jacob Russell, 
Edson Sampaio, 
Saul Schleimer, 
Carlos Servan,  
Robert Tang,
Sam Taylor,
and 
Scott Wolpert 
for helpful conversations and comments. We are especially grateful to Kasra Rafi for generously teaching us the results in Appendix \ref{A:Rafi} and for clarifying a number of his results to us. 

As mentioned above, Fred Benirschke, Ben Dozier, and John Rached have independently obtained a version of the first of our two main results \cite{BDR}; we thank them for discussing their work with us and for offering helpful comments on our work.

The second author was partially supported by  NSF Grants DMS 1856155 and 2142712 and a Sloan Research Fellowship.

\section{Cylinders via horocycle and geodesic flows}
\label{s:cyl}

\subsection{Main results.} This section lays the foundation for the rest of the paper and contains multiple results and constructions that will be used later on. 

Let $N \subseteq \mathcal{T}_{g,n}$ be an algebraic totally geodesic submanifold of Teichmüller space. Denote by $QN$ the set of quadratic differentials generating Teichm\"uller discs completely contained in $N$ (together with the zero differentials on surfaces in $N$) and let $Q_{\max} N$ denote the largest stratum of $QN$; this is well defined since $\pi(QN)$ is an irreducible variety. The set $Q_{\max} N$ is a lift of an invariant subvariety $\pi(Q_{\max} N)$ of quadratic differentials in moduli space. 

We say that a simple closed curve (on the marking surface) is the core curve of a cylinder on a quadratic differential if it is freely homotopic to the core curve of that cylinder.  We say that a simple closed curve is a cylinder curve for $QN$ if it is the core curve of a cylinder at some point of $QN$. Note that, since $QN$ is the closure of $Q_{\max}{N}$, every cylinder curve for $QN$ is the core curve of a cylinder at some point of $Q_{\max} N$. One of the main technical results of this section is the following. 

\begin{theorem}\label{T:Connected} 
Let $N \subseteq \mathcal{T}_{g,n}$ be an algebraic totally geodesic submanifold of Teichmüller space. For any cylinder curve $\alpha$ of $QN$, the subset of $Q_{\max}N$ where $\alpha$ is the core curve of a cylinder is path connected. 
\end{theorem} 

Following \cite{Wcyl}, given $q \in Q_{\max} N$, we say two cylinders on $q$ are $Q_{\max} N$-parallel if they are parallel on every quadratic differential in a sufficiently small neighborhood of $q$ in $Q_{\max}N$. We say that two cylinder curves are $N$-equivalent if there is a point of $Q_{\max} N$ where they are the core curves of $Q_{\max} N$-parallel cylinders; if $N$ is clear from context we sometimes omit it from the terminology. The significance of Theorem \ref{T:Connected} is that it allows us to prove the following result, which will be the foundation for much of our later discussion. 

\begin{corollary}\label{C:Connected}
Let $N \subseteq \mathcal{T}_{g,n}$ be an algebraic totally geodesic submanifold of Teichmüller space and let $\alpha, \beta$ be $N$-equivalent cylinder curves. Then, at every point $q$ of $QN$ where $\alpha$ is a cylinder,  $\beta$ is also a cylinder. If $q\in Q_{\max} N$, then the associated cylinders are $Q_{\max} N$-parallel. Furthermore, the ratios of moduli, heights, and circumferences of the associated cylinders are constant depending only on $\alpha, \beta$, and $N$, and, in the case of moduli, such constant is rational.
\end{corollary} 

Notice that Corollary \ref{C:Connected} guarantees the terminology $N$-equivalent is appropriate in the sense that it defines an equivalence relation among cylinder curves.

Another crucial result of this section is Lemma \ref{L:ratio}, which guarantees that the ratio of hyperbolic lengths of $N$-equivalent cylinder curves is bounded when the curves have bounded length. 

\subsection{Deducing the corollary.} To deduce Corollary \ref{C:Connected} from Theorem \ref{T:Connected} we will make use of the following result.

\begin{lemma}\label{L:ConstRatio}
Let $N \subseteq \mathcal{T}_{g,n}$ be an algebraic totally geodesic submanifold of Teichmüller space. The ratio of moduli, heights, and circumferences of different $Q_{\max} N$-parallel cylinders is locally constant, i.e., given a point $q$ of $Q_{\max} N$ and a pair of $Q_{\max} N$-parallel cylinders on it, on any sufficiently small deformation in $Q_{\max} N$, these ratios do not change. Furthermore, the ratios of moduli are rational. 
\end{lemma}

\begin{proof}
That the ratio of circumferences stays constant is an immediate consequence of the definitions \cite{Wcyl}. For the ratio of the moduli, one has to first recall \cite[Theorem 1.3]{Wri20}, which says that $Q_{\max} N$ is the lift of a particular kind of invariant subvariety, more precisely, it has no rel, and then invoke \cite[Corollary 1.6]{MW17}, which gives the desired result. (See also \cite[Lemma 6.10]{AWdiamonds} for a different proof of the result powering \cite[Corollary 1.6]{MW17}.)
\end{proof}

We are now ready to deduce Corollary \ref{C:Connected} from Theorem \ref{T:Connected}.

\begin{proof}[Proof of Corollary \ref{C:Connected}]
Let $\alpha$ and $\beta$ be $N$-equivalent cylinder curves and let $q$ be a point of $Q_{\max}N$ where they are core curves of $Q_{\max} N$-parallel cylinders. Let $q'\in Q_{\max} N$ be any point where $\alpha$ is a cylinder. Theorem \ref{T:Connected} gives a path in  $Q_{\max}N$ from $q$ to $q'$ such that $\alpha$ is a cylinder at every point on this path. By Lemma \ref{L:ConstRatio}, we get that $\beta$ must also be a cylinder at every point on this path. This argument also gives that the ratios of moduli, heights, and circumferences only depend on $\alpha, \beta$, and $N$ for differentials in $Q_{\max}$. 

Because $QN$ is the closure of $Q_{\max} N$ and because all the quantities involved are continuous, the ratios are the same on $QN$ as they are on $Q_{\max} N$. 
\end{proof}

\subsection{Cylinders via horocycle flow.}\label{SS:CylsFromHorocycleFlow} To prove Theorem \ref{T:Connected} we will need different methods to produce quadratic differentials with cylinders. We first discuss a method that relies on the Teichmüller horocycle flow
$$u_t=
\begin{pmatrix}
1 & t \\
0 & 1 
\end{pmatrix}.$$ 
Given $X \in \mathcal{T}_{g,n}$, denote by $Q(X)$ the set of quadratic differentials on $X$.

\begin{lemma}\label{L:GenerateHorocycle}
For any two distinct points $X,Y \in \mathcal{T}_{g,n}$, there exists a unique, up to scale, $q\in Q(X)$ and a unique time $t > 0$, both depending continuously on $X$ and $Y$, such that $u_t(X,q)=(Y,q')$ for some quadratic differential $q'$, and this $q$ generates the Teichm\"uller disc through $X$ and $Y$. If both $X$ and $Y$ are in a totally geodesic submanifold $N$, then $q\in QN$. 
\end{lemma}

\begin{proof}
There exists a unique positive Teichmüller horocycle arc joining $X$ to $Y$; see Figure \ref{fig:horocycle}. Furthermore, this arc belongs to the unique Teichm\"uller disc through $X$ and $Y$. Any quadratic differential $q$ as in the statement generates this Teichm\"uller disc. 
\end{proof}

\begin{figure}
    \centering
    \includegraphics[width=0.35\linewidth]{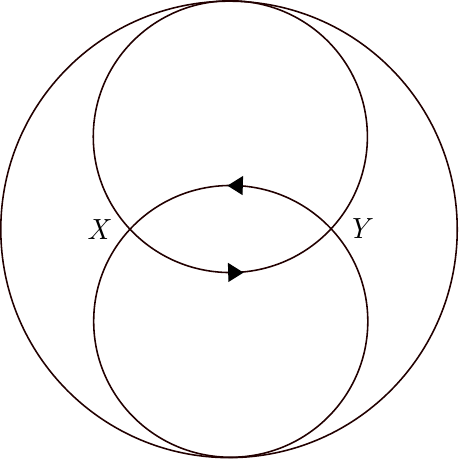}
    \caption{Horocycle arcs joining pairs of points in Teichmüller space. Arrows represent positive time directions.}
    \label{fig:horocycle}
\end{figure}

In the setting of Lemma \ref{L:GenerateHorocycle} we say that the quadratic differential $q$ generates the horocycle from $X$ to $Y$. We next recall the following. 

\begin{lemma}\label{L:StrebelGeneratingHorocycle}
Let $X \in \mathcal{T}_{g,n}$ and let $\alpha=\sum k_i \alpha_i$ be a multi-curve with all $k_i$ positive integers. Let $D_\alpha :=\prod D_{\alpha_i}^{k_i}$ be the Dehn twist in $\alpha$. Then the differential that generates the horocycle from $X$ to $D_\alpha X$ is horizontally periodic and the core curves of the horizontal cylinders are the $\alpha_i$. If $M_i$ is the modulus of the cylinder with core curve $\alpha_i$, then $M_i/M_j=k_i/k_j$. 
\end{lemma}

\begin{proof}
For any choice of positive numbers $m_i$ there exists a, unique up to scale, differential $q$ on $X$ which is horizontally periodic with core curves $\alpha_i$ of modulus $cm_i$ for some $c=c(X)>0$ \cite[Theorem 21.7]{StrebelBook}. Here we use $m_i=k_i$. It takes time $1/M$ for horocycle flow to accomplish a Dehn twist on a horizontal cylinder of modulus $M$, so let $t:= k_i / (cm_i) = 1/c$.  Then we get that $u_t q$ is a differential on $D_\alpha X$. By Lemma \ref{L:GenerateHorocycle}, $q$ generates the desired horocycle. 
\end{proof}

Next, we produce Dehn twists that stabilize a totally geodesic submanifold.  

\begin{lemma}\label{L:InvariantUnderTwist}
Let $N \subseteq \mathcal{T}_{g,n}$ a totally geodesic submanifold and let $\{\alpha_1, \ldots, \alpha_r\}$ be the core curves of an equivalence class of $Q_{\max}N$-parallel cylinders. Let $k_1, \ldots, k_r$ be positive integers such that the $\alpha_i$ are cylinder core curves on a differential in $QN$ with the corresponding moduli $M_i$ satisfying $M_i/M_k=k_i/k_j$. If $\alpha :=\sum k_i \alpha_i$, then $D_\alpha N =N$.
\end{lemma}

Morally, this lemma is true because $D_\alpha$ is in the fundamental group of $\pi(N)$ and that fundamental group stabilizes $N$, but we find it easier to give a careful proof from a slightly different perspective.

\begin{proof}
There is an open subset of $QN$ of quadratic differentials where the $\alpha_i$ are core curves of cylinders. The Cylinder Deformation Theorem then gives that $D_\alpha^{\pm1}$ applied to this subset is contained in $QN$. By analytic continuation, we deduce $D_\alpha ^{\pm1} QN \subseteq QN$. Thus, $D_\alpha QN=QN$. 
\end{proof}

We summarize this discussion as follows.

\begin{corollary}\label{C:StrebelInQN}
Let $N \subseteq \mathcal{T}_{g,n}$ be a totally geodesic submanifold. Suppose that at some point of $Q_{\max}N$ there is an equivalence class of $Q_{\max}N$-parallel cylinders with core curves $\alpha_1, \ldots, \alpha_r$ and moduli $c k_1, \ldots, ck_r$ for some $c>0$ and positive integers $k_i$. Then, for any $X \in N$, there exists a unique up to scale horizontally periodic quadratic differential $q$ on $X$ with core curves $\alpha_i$ with moduli $c' k_1, \ldots, c'k_r$ for some $c'>0$. This $q$ depends continuously on $X$ and lies in $QN$.
\end{corollary} 

\begin{proof}
Lemma \ref{L:InvariantUnderTwist} gives a multi-twist $D_\alpha$ stabilizing $N$. Using this and Lemma \ref{L:StrebelGeneratingHorocycle} one obtains the desired differential; it lies in $QN$ by Lemma \ref{L:GenerateHorocycle}.
\end{proof}

\subsection{Annular curve graphs}\label{SS:AnnularSetup} There are two versions of annular curve graphs. One is used in the hierarchically hyperbolic structure of mapping class groups and is quasi-isometric to $\bR$. The other one is used in the hierarchically hyperbolic structure of Teichm\"uller spaces and is isometric to a horoball in the hyperbolic plane. 

For the annular region corresponding to a simple closed curve $\alpha$, the mapping class group version keeps track of the twist along $\alpha$. We think of the Teichm\"uller space version as mapping to $$\{x+iy : y\geq 1\} \subseteq \bH,$$ with the $x$ coordinate again encoding the twist along $\alpha$ and the $y$ coordinate being equal to $\max(1, 1/\ell_\alpha(X))$, where $\ell_\alpha(X)$ is the hyperbolic length of $\alpha$ at $X\in \cT_{g,n}$. In particular, we note that the mapping class group version can be thought of as the $x$ coordinate of the Teichm\"uller space version. 

If desired, we can view $\bR$ as the boundary of the horoball with its intrinsic path metric, and consider the projection $x+iy \mapsto x+i$ from the horoball to its boundary. We can think of the mapping class group version of the map to the annular curve graph as the composition of the Teichm\"uller space version together with this projection. The fact that the projection from a horoball to its boundary is highly non-Lipschitz offers some hint that using the mapping class group version of the annular curve graph when discussing Teichm\"uller spaces requires extra care.  

The mapping class group version has been standard since the work of Masur-Minsky \cite[Section 2.4]{MMII} and the Teichm\"uller space version was foreshadowed in Minsky's product region theorem \cite{MinskyProduct}, is implicit in formulas of Rafi \cite{RafiComb}, and is explicit in subsequent work \cite{EMRrank, DurhamAugmented}. We suggest \cite[Section 3.1]{EMRrank} as a reference on this topic. 

By default we will use the Teichm\"uller space version and we will comment when using the mapping class group version. 

\subsection{Cylinders via geodesic flow.} 
Next, we discuss a coarser method for producing quadratic differentials with cylinders, via the Teichmüller geodesic flow. Let $d_\alpha^\MCG(X,Y)$ denote the distance between the subsurface projections of $X,Y \in \mathcal{T}_{g,n}$ in the mapping class group version of the annular curve graph of $\alpha$ and let $\ell_\alpha(X)$ denote the hyperbolic length of the geodesic representative of $\alpha$ on $X \in \mathcal{T}_{g,n}$. Roughly speaking, the following result guarantees the existence of cylinders whenever there is large enough twist. 

\begin{lemma}\label{L:RafiCyl}
There exists a constant $D > 0$ depending only on $g$ and $n$ such that for any $X,Y \in \mathcal{T}_{g,n}$ and every simple closed curve $\alpha$, if 
\[
d_\alpha^\MCG(X,Y) > D(1 + \ell_X(\alpha)^{-1} + \ell_Y(\alpha)^{-1}),
\]
then the quadratic differential generating the Teichm\"uller geodesic segment from $X$ to $Y$ has a cylinder with core curve $\alpha$. 
\end{lemma}

\begin{proof}
This can be thought of as a black box coming from work of Rafi, but we briefly outline the structure of the argument. 

Denote by $\nu^+$ and $\nu^-$ the vertical and horizontal foliations of the quadratic differential $q$ generating the Teichm\"uller geodesic segment from $X$ to $Y$. Since the ensuing argument gives a lower bound on the modulus of the cylinder, we may take an arbitrarily small perturbation of $X$ or $Y$, allowing us to assume that $\alpha$ has non-zero intersection number with the $\nu^\pm$. 

As we recall briefly in Appendix \ref{A:Rafi}, Rafi has an alternative definition of twisting along a Teichm\"uller geodesic which makes use of the flat metric of the generating quadratic differential \cite[Section 4]{RafiComb}. By \cite[Theorem 4.3]{RafiComb}, it is possible to compare differences in this alternative twisting to $d_\alpha^\MCG(X,Y) $ with an error of the form  $O(\ell_X(\alpha)^{-1} + \ell_Y(\alpha)^{-1})$. Hence, in our situation, there is large change in Rafi's alternative twist between $X$ and $Y$. Formulas in \cite[Section 4]{RafiComb}, which are recalled in Appendix \ref{A:Rafi}, imply the existence of a large modulus cylinder (at some point between $X$ and $Y$).
\end{proof}

\subsection{Hyperbolic lengths.} Later we will need to estimate hyperbolic lengths of short curves. We recall some of the necessary background for this now. 

\begin{lemma} \label{L:expanding}
Let $q$ be a horizontally periodic quadratic differential such that the ratios of circumferences of horizontal cylinders is at most $M$. Then, there exists a uniform bound, depending only on $M$ (and, as always,  $g$ and $n$), for the modulus of the expanding annuli corresponding to the core curves of the horizontal cylinders of $q$. 
\end{lemma}

Expanding annuli appear in work of Minksy and Rafi \cite{MinskyProduct, ThickThin} and are now a basic tool. Their definition admits slight variations which are not important here, but, for concreteness, let us say we follow \cite[Section 3]{RafiHyp}, so the inner boundary is the flat geodesic representative.

\begin{proof}
Notice that, for any such expanding annulus of large modulus, there exists an embedded horizontal segment crossing it that is of size proportional to the size of the outer boundary. The circumference of one of the cylinders under consideration is as big as this horizontal segment. On the other hand, the length of the inner boundary is by definition a circumference in the case under consideration. Thus, the desired bound follows from the fact that, as recalled in \cite[Section 3]{RafiHyp}, the moduli of expanding annuli can be calculated as the logarithm of the ratio of the lengths of the inner and outer boundaries.
\end{proof}

The ratio assumption in Lemma \ref{L:expanding} is crucial, as seen by considering a  one by one square torus with a small circumference cylinder glued in along a horizontal slit to give a surface in $\cH(2)$. 

We deduce the following consequence of Lemma \ref{L:expanding}; here we denote the modulus of a cylinder $F$ by $\mathrm{Mod}(F)$.

\begin{corollary} \label{C:length}
    There exists $\delta>0$ (depending only on $g$ and $n$) such that for all $M>0$ there exists an $E>0$ such that the following holds. 
    Let $(X,q)$ be a horizontally periodic quadratic differential such that the ratio of circumferences of horizontal cylinders on $q$ is at most $M$. 
    Suppose $\alpha$ is the core curve of a horizontal cylinder $F$ on $q$ with $\ell_\alpha(X) \leq \delta$. Then, 
    \[
    1/E \leq \ell_\alpha(X) \cdot \mathrm{Mod}(F) \leq E.
    \]
\end{corollary}

%
%

\begin{proof}
The ratio of the inverse extremal length of $\alpha$ and the sum of $\mathrm{Mod}(F)$ and the moduli of the corresponding expanding annuli is uniformly bounded; see  \cite[Theorem 3.1, Statement (2)]{RafiHyp} for a convenient reference and \cite[Theorem 4.6]{MinskyHarmonic} for the original. Lemma \ref{L:expanding} gives that the moduli of the expanding annuli are bounded, and, in particular, if $\delta > 0$ is small enough, this implies that $F$  has large modulus. 

%
%

With the assumption that $\ell_\alpha(X) \leq \delta$, hyperbolic and extremal lengths have bounded ratios \cite[Proposition 1, Corollary 3]{maskit}, so the result follows. 
\end{proof}

\subsection{Proof of path connectivity.} We are now ready to prove Theorem \ref{T:Connected}.

\begin{proof}[Proof of Theorem \ref{T:Connected}]
    Suppose $\alpha$ is the core curve of a cylinder $C$ on $(X,q) \in Q_{\max}N$. Denote by $[C]$ the $Q_{\max}N$-parallelism class of $C$. Suppose also that $\alpha$ is the core curve of a cylinder $D$ on $(X',q') \in Q_{\max}N$. Our goal is to produce a path connecting $q$ to $q'$ in $Q_{\max} N$ along which $\alpha$ is always a cylinder.
    
    Consider the reparametrized Teichmüller geodesic $\mathcal{G} \colon [0,1] \to N$ from $X$ to $X'$. After rotating $q$ and $q'$, we can assume without loss of generality that the cylinders $C$ and $D$ are horizontal.  Applying the Teichmüller horocycle flow to $q$ and $q'$ for long enough times yields Riemann surfaces $Y \in N$ and $Y' \in N$ whose projections to the annular curve graph of $\alpha$ are coarsely equal and very far from the projection of any point on $\mathcal{G}$. Furthermore, because the Teichmüller horocycle flow is strongly nondivergent \cite{nondiv}, it is possible to arrange for $X,Y,X',Y' \in N$ to be $\epsilon$-thick for some $\epsilon := \epsilon(X,Y) > 0$ independent of the constant implicit in the notion of ``very far" above. 
    %
    %
    %

    \begin{figure}[h!]
    \includegraphics[width=0.45\linewidth]{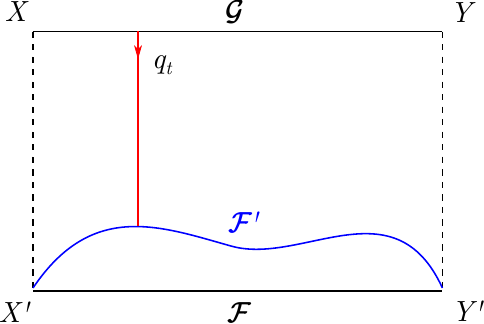}
    \caption{Dotted lines are Teichmüller horocycles and solid straight lines are Teichmüller geodesics.}
    \label{F:HoroGeo}
    \end{figure} 
    
    Let $\mathcal{F} \colon [0,1] \to N$ denote the  reparametrized Teichmüller geodesic from $X'$ to $Y'$. Now, by Corollary \ref{C:StrebelInQN}, for every $Y'' \in \mathcal{F}$ we can find a continuously varying horizontally periodic quadratic differential on $Y''$ whose horizontal cylinder core curves are exactly those corresponding to cylinders in $[C]$. Using the Cylinder Deformation Theorem over these differentials together with Corollary \ref{C:length}, it is possible to produce a continuous path $\mathcal{F}' \colon [0,1] \to N$ with the same endpoints as $\mathcal{F}$ and such that for every $Y'' \in \mathcal{F}'$ it holds that $\ell_{\alpha}(Y'') > \epsilon'$ and that the projection of $Y''$ to the annular curve graph of $\alpha$ is very close to that of $X'$ and $Y'$. Here $\epsilon'$ is a positive constant depending on $\epsilon$ and the constant in Corollary \ref{C:length}. See Figure \ref{F:HoroGeo} for a picture of this construction. 

    Now let $q_t \colon [0,1] \to QN$ be the continuous one parameter family of quadratic differentials such that $q_t \in QN$ is the initial quadratic differential of the Teichmüller geodesic segment joining $\mathcal{G}(t) \in N$ to $\mathcal{F}'(t) \in N$. Since the Teichmüller horocycle flow preserves Teichm\"uller disc, $q_0$ can be obtained as a rotation of $q$ and $q_1$ can be obtained as a rotation of $q'$. Originally $C$ was defined as a cylinder on $q$, but, by abuse of notation, we use the same symbol to denote the corresponding cylinder on $q_0$, since $q$ and $q_0$ only differ by a rotation. 

    Recall that $QN$ is a vector bundle over $N$. The non-maximal part $QN \setminus Q_{\max} N$ is the union of the intersections of $QN$ with certain lower dimensional strata. It is thus possible (for example smoothing and using Sard's Theorem) to produce arbitrarily small deformations $\hat{q}_t \colon [0,1] \to Q_{\max} N$ of the path $q_t$ with image contained in $Q_{\max} N$. Since the endpoints of $q_t$ are already in $Q_{\max} N$, we can additionally arrange for $\hat{q}_t$ to agree with $q_t$ at the endpoints. 

    Now, by Lemma \ref{L:RafiCyl}, assuming the path $\hat{q}_t$ is sufficiently close to the path $q_t$, we get that $\alpha$ must be a cylinder on $\hat{q}_t$ for all $t \in [0,1]$. We have thus produced the desired path connecting $q$ and $q'$ in $Q_{\max} N$.
\end{proof}

\subsection{Short curves and cylinders.} The following lemma and subsequent corollary emphasize the important connection between short curves and cylinders.

\begin{lemma}\label{L:GetCyls}
Let $N \subseteq \mathcal{T}_{g,n}$ be an algebraic totally geodesic submanifold. Then, there exists some $\epsilon>0$ depending only on $N$ such that if $\{\alpha_1, \ldots, \alpha_r\}$ is a collection of disjoint simple closed curves that can be made simultaneously of hyperbolic length less than $\epsilon$ on some surface in $N$, then there exists a quadratic differential in $QN$ where they are all core curves of cylinders. 
\end{lemma}

\begin{proof}
A compactness argument using the Deligne--Mumford compactification provides an $\epsilon>0$ such that if $\{\alpha_1, \ldots, \alpha_r\}$ is a collection of disjoint simple closed curves that can be made simultaneously of hyperbolic length less than $\epsilon$ on some surface in $N$, then the boundary of $N$ intersects a stratum of the Deligne--Mumford bordification where all these $\alpha_i$ are pinched to nodes. 

Working in moduli space rather than Teichm\"uller space,  considering a holomorphic disc through such a boundary point, and  using plumbing coordinates, one can check that  $N$ is stabilized by a multi-twist whose support includes the $\alpha_i$; see \cite[Proposition 3.1]{Equations} and \cite[Proposition 7.6]{FredBoundary} for details. Corollary \ref{C:StrebelInQN} then gives the result. Alternatively, after applying a large power of this multi-twist to any point in $N$ and considering the corresponding quadratic differential, Lemma \ref{L:RafiCyl} gives the desired result. 
\end{proof}

\begin{corollary}\label{C:Pinchable}
Let $N \subseteq \mathcal{T}_{g,n}$ be an algebraic totally geodesic submanifold. Then, the core curves of cylinders of quadratic differentials in $QN$ are exactly the curves that get arbitrarily short on $N$. 
\end{corollary}

\begin{proof}
The hyperbolic length of the core curve of any cylinder can be made arbitrarily small using the Teichm\"uller geodesic flow or the Cylinder Deformation Theorem. Conversely, if $\alpha$ is a curve that gets arbitrarily short on $N$, Lemma \ref{L:GetCyls} gives that $\alpha$ is a cylinder curve for $QN$.
\end{proof}

It will be very useful for our purposes to control the ratios of hyperbolic lengths of short curves, as guaranteed by the following lemma.

\begin{lemma}\label{L:ratio}
    Let $N \subseteq \mathcal{T}_{g,n}$ be an algebraic totally geodesic submanifold. Suppose $\{\alpha_1,\dots,\alpha_r\}$ is an equivalence class of $N$-equivalent simple closed curves. Then, there exist constants $M_0 > 0$ and $C > 0$ depending only on $N$ such that for every $i,j \in \{1,\dots,r\}$ and every $X \in N$, if $\min(\ell_{\alpha_i}(X), \ell_{\alpha_j}(X)) \leq M_0$, then
    \[
    1/C \leq \ell_{\alpha_i}(X)/\ell_{\alpha_j}(X) \leq C.
    \]
\end{lemma}

\begin{proof}
Corollary \ref{C:StrebelInQN} produces a differential on $X$ in $QN$ with horizontal cylinders corresponding to the $\alpha_i$ and whose ratio of moduli do not depend on $X$.
Corollary \ref{C:length} allows the hyperbolic lengths to be estimated in terms of these moduli, keeping in mind that once one of the $\alpha_i$ gets short the corresponding modulus must be large, hence all the moduli must be large and all the other $\alpha_j$ must be short as well. 
\end{proof}

\subsection{Simultaneous realization.} For later use we highlight the following results. 

\begin{lemma}\label{L:Simultaneous}
Let $N \subseteq \mathcal{T}_{g,n}$ be an algebraic totally geodesic submanifold. Let $\alpha$ and $\beta$ be cylinders at possibly different points of $QN$. Then there exists $q\in QN$ on which both $\alpha$ and $\beta$ are cylinders. 
\end{lemma}

\begin{proof}
Fix a point $X\in N$ and find points $Y, Z \in N$ such that
\begin{gather*}
    \ell_\alpha(X) = \ell_\alpha(Y) \quad \text{and} \quad \ell_\beta(X) = \ell_\beta(Z),\\
    d_\alpha(X,Y)>K \quad\text{and} \quad d_\beta(X,Z)>K,
\end{gather*}
where $K > 0$ is a very large constant. This can be done, as in previous proofs, by applying large powers of appropriate multi-twists. 

Now, suppose in first instance that $$d_\alpha(Y,Z) > K/100 \quad\text{and} \quad d_\beta(Y,Z)> K/100.$$
Then, by Lemma \ref{L:RafiCyl}, the quadratic differential generating the Teichm\"uller geodesic segment from $Y$ to $Z$ proves the desired result.

Thus, it suffices by symmetry to consider the case when $d_\alpha(Y,Z) \leq  K/100$. In this case, 
$d_\alpha(X, Z)> K/2$ by the triangle inequality, so, again by Lemma \ref{L:RafiCyl}, the quadratic differential generating the Teichmüller geodesic segment from $X$ to $Z$ proves the desired result. 
\end{proof}

Notice that, in the context of Lemma \ref{L:Simultaneous}, it follows by Corollary \ref{C:Connected} that all simple closed curves $N$-equivalent to $\alpha$ or $\beta$ are core curves of cylinders on $q$. 

\begin{corollary}\label{C:Simultaneous}
Let $N \subseteq \mathcal{T}_{g,n}$ be an algebraic totally geodesic submanifold. Then, any two $N$-equivalence classes of simple closed curves are either pairwise disjoint or every curve of one intersects a curve of the other. 
\end{corollary}

\begin{proof}
Pick a curve in each equivalence class. Lemma \ref{L:Simultaneous} provides a quadratic differential $q\in QN$ where both curves are cylinders. Perturbing $q$, without loss of generality we can assume that $q$ is in $Q_{\max}N$. Corollary \ref{C:Connected} gives that all the $N$-equivalent curves are cylinder curves on $q$ as well.  Let $[C]$ and $[D]$ denote the corresponding equivalence classes of cylinders on $q$. 

Using the Cylinder Deformation Theorem, we can deform $[C]$ in such a way that the circumference of every cylinder in $[D]$ that intersects a cylinder in $[C]$ changes. Since ratios of circumference of cylinders in $[D]$ are constant, this gives that if any cylinder of $[D]$ is intersected by a cylinder in $[C]$, then they all are. See \cite[Proposition 3.2]{H4} for details.
\end{proof}

\section{The boundary is totally geodesic}

\subsection{Main definition and theorem.} 
Throughout this section we consider a product of Teichm\"uller spaces of the form
$$\mathcal{T}_{g_1,n_1}\times \cdots \times \mathcal{T}_{g_s,n_s}.$$
We may sometimes use the Teichm\"uller metric on this space, which can be defined in the usual way and is equal to the sup of the Teichm\"uller metrics of the components. We will use $\pi$ to denote the map to the corresponding product of moduli spaces. 

Over these product spaces we consider the corresponding bundles of quadratic differentials holomorphic away from marked points and with at most simple poles at marked points; these differentials have finite area and could be zero on some (or all) components. There is an action of $GL^{+}(2,\mathbb{R})$ on this bundle which acts as usual on components where the differential is non-zero and leaves unchanged components where the differential is zero. We warn the reader that this action is not continuous at points where the differential is zero on at least one component. 

We define a T-disc as the projection to $\mathcal{T}_{g_1,n_1}\times \cdots \times \mathcal{T}_{g_s,n_s}$ of the $GL^{+}(2,\mathbb{R})$ orbit of one of the quadratic differentials introduced above.   This term of course stands for ``Teichm\"uller disc'', but we have chosen to instead write ``T-disc'' to emphasize this is a non-standard definition given the multi-component setting. The T-disc defined using a quadratic differential $q$ will be called the T-disc generated by $q$.

\begin{remark}\label{R:Tdiscs}
We highlight a few important aspects of T-discs:
\begin{enumerate}
\item The T-disc associated to a zero differential is a point. 
\item The projection to each component $\cT_{g_i, n_i}$ of any T-disc is either a point or a Teichm\"uller disc. 
\item\label{R:Tdiscs:equal} If $(X_1, \ldots, X_s)$ and $(Y_1, \ldots, Y_s)$ lie on single T-disc, then, for each pair $i,j$ of coordinates where the disc is not constant, we have 
$$d_{\cT_{g_i, n_i}} (X_i, Y_i) = d_{\cT_{g_j, n_j}} (X_j, Y_j).$$
\item Every T-disc is an isometrically and holomorphically embedded copy of the hyperbolic plane.
\item Not every isometrically and holomorphically embedded copy of the hyperbolic plane is a T-disc. 
\item It is not true that there is a T-disc containing every pair of points in a product of Teichmüller spaces.
\item\label{R:Tdiscs:notUnique} If $q$ and $q'$ are equal up to rescaling on each component by a non-zero factor depending on the component, then the T-discs they define are the same. In particular, when the number of components is at least 2, it is possible to have $\bC q \neq \bC q'$ but for $q$ and $q'$ to generate the same T-disc. 
\end{enumerate}
\end{remark}

Given an irreducible complex analytic subset
$$L \subseteq \mathcal{T}_{g_1,n_1}\times \cdots \times \mathcal{T}_{g_s,n_s},$$
we define $QL$ to be the set of quadratic differentials that generate T-discs completely contained in $L$.

\begin{definition}\label{D:TG}
An irreducible complex analytic subset $L \subseteq \mathcal{T}_{g_1,n_1}\times \cdots \times \mathcal{T}_{g_s,n_s}$ is said to be totally geodesic if there is an irreducible complex analytic subset $Q_{\rm alg}L$ of the bundle of quadratic differentials over $\mathcal{T}_{g_1,n_1}\times \cdots \times \mathcal{T}_{g_s,n_s}$ such that 
\begin{enumerate}
\item $Q_{\rm alg}L \subseteq QL$ and $Q_{\rm alg}L$ is $GL^+(2, \bR)$ invariant, 
\item $\pi(Q_{\rm alg} L)\subseteq Q\mathcal{M}_{g_1,n_1}\times \cdots \times Q\mathcal{M}_{g_s,n_s}$ is a variety, and 
\item $Q_{\rm alg} L$ contains a $\dim L$ dimensional vector subspace above every point of $L$. 
\end{enumerate}
\end{definition}

\begin{remark}
Recall that a complex analytic subset is a subset that is locally defined near each point of the ambient space as the zero set of a finite collection of holomorphic functions.  Such subsets are always closed, and we emphasize that this means $Q_{\rm alg}L$ is required to be closed. 
\end{remark}

\begin{remark}
Later, after we understand their remarkable structure, in Lemma \ref{L:smooth} we will see that totally geodesic analytic subsets are actually complex submanifolds, i.e., have no singularities. So the appearance of complex analytic subsets can be thought of as a technical detail internal to our proofs. The reader wishing to pay attention to this detail may want to keep in mind that there are complex analytic subsets which are topological submanifolds but not complex submanifolds; for example, consider the subset of $\bC^2$ defined by $z_1^2=z_2^3$.
\end{remark}

\begin{remark}
The whole product space $\mathcal{T}_{g_1,n_1}\times \cdots \times \mathcal{T}_{g_s,n_s}$ is totally geodesic, but, as remarked above, it does not contain a T-disc through every pair of points. Thus, our definition does not satisfy all the conditions one might initially hope from a generalized definition of totally geodesic.  
\end{remark} 

\begin{remark}\label{R:2dim}
As is discussed in \cite[Proposition 1.3]{GoujardSurvey} and elsewhere, in the single component case, $L$ is totally geodesic if and only if $\dim QL = 2\dim L$.   Since \cite[Proposition 3.1]{IsomEmb} shows that in this case $QL$ is a vector bundle, it is not hard to see that our definition agrees with the usual definition when $s=1$; see Lemma \ref{L:smooth} for why $L$ is a complex submanifold and not merely a complex analytic subset.
\end{remark}

\begin{remark}
In the single component case one does not always consider zero dimensional totally geodesic submanifolds, i.e., single points, but we note that, for our purposes, a single point will be considered as such.
\end{remark}

\begin{remark}
In the multi-component case, $QL$ is often bigger than $Q_{\rm alg} L$. For example, $QL$ is closed under scaling each component individually by positive real numbers, and need not be closed in the bundle of quadratic differentials. This behavior can already be seen for a diagonal. 
\end{remark}

Definition \ref{D:TG} is easy enough to satisfy that we can verify the following, which is the main result of this section.

\begin{theorem}\label{T:BoundaryTotallyGeodesic}
Let $N\subseteq \cT_{g,n}$ be an algebraic totally geodesic submanifold. Then any irreducible component of the intersection of its closure $\overline{N}$ with a stratum of the Deligne--Mumford bordification is an algebraic totally geodesic subset.
\end{theorem}

Note that Definition \ref{D:TG} actually assumes algebraicity, so when in the phrase ``algebraic totally geodesic subset" above the word ``algebraic" is really just for emphasis and for consistency with other situations where algebraicity is not automatic. 

We emphasize that Definition \ref{D:TG} should be thought of as a temporary tool, since ultimately in proving Theorem \ref{T:main1} we will obtain vastly stronger information. 

\subsection{The quadratic Hodge bundle.} Following the notation of, for example, \cite{CMW}, 
we consider the quadratic Hodge bundle over the Deligne--Mumford compactification $\ol{\cM}_{g,n}$; see also \cite[Sections 3.1, 3.2]{StrataOfkDiffs}. This holomorphic vector bundle parameterizes quadratic differentials with at worst simple poles at the marked points and at worst double poles at the nodes, subject to residue matching condition at nodes; these differentials are allowed to be zero on some components. We also consider the pull-back of this bundle to the Deligne--Mumford bordification, leaving it to the reader to infer from context if we are referring to the lifted or the original version. (The pullback is a vector bundle, see for instance \cite[Proposition 1.5]{HatcherVB}, but since the bordification is not even locally compact it is not a holomorphic vector bundle). 

The residue matching condition is equality of the coefficients of $1/z^2$ in the Laurent expansions of the differential on either side of the node; following the work cited above we call this coefficient the 2-residue. When the differential is viewed as a flat surface, having non-zero and matching 2-residues corresponds to having half-infinite cylinders on either side of the node with core curves that are parallel and of the same length. 

Let $N\subseteq \cT_{g,n}$ be an algebraic totally geodesic submanifold. Denote by $\ol{Q N}$ the closure of $QN$ in the quadratic Hodge bundle. Given two $N$-equivalent cylinder curves $\alpha$ and $\beta$, denote by $c(\alpha, \beta)$ the circumference of a cylinder with core curve $\alpha$ divided by the circumference of a cylinder with core curve beta; this quantity is well defined  by Corollary \ref{C:Connected}.

\begin{lemma}\label{L:ResidueEquation}
Let $N\subseteq \cT_{g,n}$ be an algebraic totally geodesic submanifold and $\alpha$, $\beta$ be a pair of $N$-equivalent cylinder curves. Then, for any $q\in \ol{Q N}$ where $\alpha$ and $\beta$ are pinched to nodes $n_\alpha$, $n_\beta$, the 2-residue of $q$ at $n_\alpha$ is equal to $c(\alpha, \beta)$ times the 2-residue of $q$ at $n_\beta$.
\end{lemma}

\begin{proof}
This follows by continuity of the circumference of a cylinder and the fact that 2-residues give the circumferences of the corresponding infinite cylinders. 
\end{proof}

\begin{corollary}\label{C:ManyFiniteArea}
Let $N\subseteq \cT_{g,n}$ be an algebraic totally geodesic submanifold and suppose $X\in \ol{N}$ is a point where exactly $k$ classes of $N$-equivalent curves have been pinched to nodes. Then, the fiber of $\ol{Q N}$ at $X$ contains a linear subspace of dimension $d-k$ consisting of finite area differentials. Moreover any finite area differential in the fiber is contained in such a subspace. 
\end{corollary} 

\begin{proof}
Recall from \cite[Proposition 3.1]{IsomEmb} that $QN$ is a vector bundle of rank $d=\dim N$. Thus, every fiber of $\ol{Q N}$ is a union of $d$-dimensional vector subspaces $W$  of differentials. Given such a $W$, Lemma \ref{L:ResidueEquation} gives that the subspace $W_0$ of $W$ where all 2-residues vanish has codimension at most $k$. This $W_0$ is exactly the finite area differentials in $W$.
\end{proof}

We deduce the following corollary from \cite{MW17, ChenWright}. 

\begin{corollary}\label{C:NotJustContinuity}
Let $V$ be a $GL^{+}(2,\mathbb{R})$-invariant subvariety of non-zero quadratic differentials over $\cM_{g,n}$. Then the closure of $V$ in the finite area part of the quadratic Hodge bundle is $GL^{+}(2,\mathbb{R})$ invariant. 
\end{corollary}

Because the $GL^{+}(2,\mathbb{R})$ action on the closure of the quadratic Hodge bundle is not continuous, this is a non-trivial statement. Although restricting to the finite area part is not required, we will only need that case here.

\begin{proof}
For concreteness of notation we will prove invariance under the one parameter subgroup $\{u_t\}_{t \in \mathbb{R}}$ consisting of unipotent upper triangular matrices. This is sufficient because the closure of $V$ in the finite area part of the quadratic Hodge bundle is easily seen to be invariant under multiplication by complex scalars, and complex scalar multiplication together with $\{u_t\}_{t \in \mathbb{R}}$ generate $GL^{+}(2,\mathbb{R})$. 

Consider a sequence $(X_k, q_k) \in V$ that converges to $(X_\infty, q_\infty)$ in the finite area part of the quadratic Hodge bundle. We show $u_t(X_\infty, q_\infty)$ is also in the closure of $V$. Without loss of generality, assume that all $(X_k,\omega_k)$ lie in a single stratum and in a single irreducible component of the intersection of $V$ with that stratum. 

Let $(Y_k, \omega_k)$ be the double cover of $(X_k, q_k)$. After passing to a subsequence, we can assume $(Y_k, \omega_k)$ converges in the Hodge bundle to a finite area limit $(Y_\infty, \omega_\infty)$. Let $Y_\infty'$ (respectively $X_\infty'$) be the union of the components of $Y_\infty$ (respectively $X_\infty$) where $\omega_\infty$ (respectively $q_\infty$) is nonzero. Then, $(Y_\infty', \omega_\infty)$ is a double cover of $(X_\infty', q_\infty)$.

The construction above guarantees $(Y_k, \omega_k)$ converges to $(Y_\infty', \omega_\infty)$ in the ``What You See Is What You Get" partial compactification of the stratum. The main result of \cite{MW17, ChenWright} gives that, for any deformation on the boundary in this partial compactification, a corresponding deformation can be done on $(Y_k, \omega_k)$ for $k$ sufficiently large. The corresponding deformation does not change the vanishing cycles, so, at an intuitive level, one should think that the deformation does not affect the part of $Y_k$ where $\omega_k$ is converging to zero. (Results in \cite[Section 9]{MW17}  relate the topology on the What You See Is What You Get to the topology on the Hodge bundle, and allow vanishing cycles to be related to subsurfaces where $\omega_k$ is converging to zero.)

Applying this to the $u_t$ orbit of $(Y_\infty', \omega_\infty)$, which in period coordinates is given by $\omega_\infty + t \Im(\omega_\infty)$, shows that the path $\omega_k + t \Im(\omega_\infty)$ is contained in the set of double covers of surfaces in $V$, for $k$ sufficiently large and $t$ sufficiently small. For fixed $t$ small, this can be seen to converge to $u_t(Y_\infty, \omega_\infty)$. For more details, compare to \cite[Section 9]{MW17} and \cite[Section 8]{ChenWright}. This shows that for $t$ small $u_t(X_\infty, q_\infty)$ is in the closure of $V$ as desired. Since the closure is closed, proving the result for $t$ sufficiently small depending on the limit point is sufficient. 
\end{proof}

\subsection{Algebraicity} We will  need the following basic algebraicity result. 

\begin{lemma}\label{L:CoverVariety}
Let $V$ be a  subvariety of $\cM_{g,n}$ and $\smash{\widehat{V}}$ be an irreducible component of its preimage in $\cT_{g,n}$. Let $B$ be the closure of $\smash{\widehat{V}}$ in the Deligne--Mumford bordification intersected with a stratum of the bordification. Then, the image of $B$ in the Deligne--Mumford compactification $\ol\cM_{g,n}$ is a variety.
\end{lemma}

\begin{proof}
We will use the well studied construction of Dehn space; see \cite{BersBull} for details and \cite[Section 5.2]{BainbridgeThesis} for a more recent introduction. We follow the notation of the later reference, so $S$ is a set of curves that are pinched to obtain a stratum of the Deligne--Mumford bordification, $\Tw(S)$ is the group generated by Dehn twists along curves of $S$, and $\Stab(S)$ is the stabilizer of the set $S$ in the whole mapping class group. Note that $\Tw(S)$ is a normal subgroup of $\Stab(S)$.

The Dehn space $\cD_{g,n}(S)$ associated to $S$ can be viewed as arising from a two step construction: first take Teichm\"uller space union  the stratum of the Deligne--Mumford bordification corresponding to $S$ and then quotient by $\Tw(S)$. The key results we need are: the map from $\cD_{g,n}(S)$ to $\ol\cM_{g,n}$ is invariant under the induced action of $\Stab(S)/\Tw(S)$; stabilizers for this action are finite; and up to those finite stabilizers the map is a local homeomorphism. 

We will also make use of the following topological fact:  Given a sequence $Z_n$ in Dehn space converging to a boundary point $Z$, for any choice of preimages $Z_n'$  of the $Z_n$ in Teichm\"uller space, the sequence $Z_n'$ converges to the unique point $Z'$ (in the appropriate stratum of the Deligne--Mumford bordification) that maps to $Z$. (One can keep in mind that the different preimages of $Z_n$ all differ by $\Tw(S)$, and applying elements of $\Tw(S)$ does not effect convergence to points where the curves in $S$ are pinched.) 

Let $\smash{\widetilde{V}}$ be the image of $\smash{\widehat{V}}$ in Dehn space, so we have 
$$\smash{\widehat{V}} \to \smash{\widetilde{V}} \to V.$$
We first note that the topological fact above implies that the boundary of $\smash{\widehat{V}}$ in the stratum corresponding to $S$ covers the boundary of $\smash{\widetilde{V}}$. Thus it suffices to show that the image of the boundary of $\smash{\widetilde{V}}$  is a variety in $\ol\cM_{g,n}$. 

Consider $\cM_{g,n}$ union the boundary stratum corresponding to $S$, and let $\cU$ be a neighborhood of the boundary. Up to passing to a finite cover, this $\cU$ can be chosen so that it can be parametrized by a product of moduli spaces (one for each component after the curves in $S$ are pinched) and punctured discs (one for each curve in $S$). Viewed as a complex analytic space, we can arrange for the intersection $\cU\cap V$ to have finitely many irreducible components. 
The boundary of each of these components is a variety. 

Let $\smash{\widetilde{\cU}}$ be the component of the preimage of $\cU$ in Dehn space whose closure intersects the boundary. This can be taken to be parametrized by a product of Teichm\"uller spaces (one for each component after the curves in $S$ are pinched) and punctured discs (one for each curve in $S$).

Because $\smash{\widetilde{V}}$ is an irreducible component of the preimage of $V$ in Dehn space and because the map from Dehn space to $\ol\cM_{g,n}$ is locally a homeomorphism up to the action of finite groups, the image of $\widetilde{V}\cap \smash{\widetilde{\cU}}$ in $\cU$ is equal to a finite union of components of  $V \cap \mathcal{U}$, and the boundary of $\smash{\widetilde{V}}$ maps onto the union of the boundaries of these components. 
\end{proof}

We also note two corollaries  of the analysis above, in whose proofs we continue to use the notation above. 

\begin{corollary}\label{C:algebraic1}
Let $V$ be a  subvariety of $\cM_{g,n}$ and $\smash{\widehat{V}}$ be an irreducible component of its preimage in $\cT_{g,n}$. Then, up to the action of the stabilizer of $\smash{\widehat{V}}$ in the mapping class group, there are only finitely many orbits of strata in the Deligne--Mumford bordification that intersect the closure of $\smash{\widehat{V}}$. 
\end{corollary}  

\begin{proof}
There are only finitely many mapping class group orbits of strata of the Deligne--Mumford bordification, so it suffices to prove finiteness in each such orbit. 

Consider a stratum that intersects the closure of $\smash{\widehat{V}}$. Near that stratum, there must be a subset of $\smash{\widehat{V}}$ that maps onto one of the finitely many components of $\cU\cap V$. Given a second stratum in the same mapping class group orbit where one similarly gets a subset mapping onto the same component, there must be a mapping class $g$ taking a germ of one subset to a germ of the other. That mapping class must stabilize $\smash{\widehat{V}}$ because $\smash{\widehat{V}}$ and $g(\smash{\widehat{V}})$ are both irreducible varieties and share a germ, so $\smash{\widehat{V}}=g(\smash{\widehat{V}})$. 
\end{proof}

\begin{corollary}\label{C:algebraic2}
Let $V$ be a  subvariety of $\cM_{g,n}$ and $\smash{\widehat{V}}$ be an irreducible component of its preimage in $\cT_{g,n}$. Let $S$ be a disjoint collection of curves. Let $\Stab(\smash{\widehat{V}})$ denote the stabilizer of $\smash{\widehat{V}}$ in the mapping class group, and $$\Tw_S(\smash{\widehat{V}})= \Stab(\smash{\widehat{V}})\cap \Tw(S).$$
Then the map from $\smash{\widehat{V}}/ \Tw_S(\smash{\widehat{V}})$ to $\cT_{g,n}/\Tw(S)$ is proper.
\end{corollary}

\begin{proof}
The image of this map is $\smash{\widetilde{V}}$. At each point the germ of $\smash{\widetilde{V}}$ is a finite union of irreducible germs. As in the proof of the previous corollary, if two germs of $\smash{\widehat{V}}$ at different points are related by a mapping class, that mapping class must stabilize $\smash{\widehat{V}}$. It follows that any point of $\smash{\widetilde{V}}$ has a compact neighborhood whose preimage in  $\smash{\widehat{V}}/ \Tw_S(\smash{\widehat{V}})$ is compact. 
\end{proof}

\subsection{Unpinching} To complete the proof of Theorem \ref{T:BoundaryTotallyGeodesic} we also need a result that allows us to unpinch nodes in a given equivalence class. We will deduce this as a corollary of the following result, which is stronger than needed at the moment but will also be useful later. 

\begin{lemma}\label{L:DilationLemma}
Let $N\subseteq \cT_{g,n}$ be an algebraic totally geodesic submanifold of Teichmüller space. For any point $X$ in the boundary of $\ol{N}$ it is possible to find a quadratic differential $q\in QN$ with a collection of disjoint equivalence classes of cylinders such that replacing each  cylinder in this collection with a pair of half-infinite cylinders gives a differential on $X$. 
\end{lemma}

\begin{proof}
Consider a sequence $X_k$ in $N$ converging to $X$. For each equivalence class of nodes of $X$ we get an associated multi-twist and from Lemma \ref{L:StrebelGeneratingHorocycle} we also get horizontally periodic differentials $\eta_k$ on $X_k$ where the horizontal cylinders correspond to the given equivalence class. By Corollary \ref{C:length}, up to rescaling and possibly passing to a subsequence, these converge in the quadratic Hodge bundle to a differential on $X$ with a pair of half-infinite cylinders at each node in the equivalence class. 

Taking the sum of one such differential for each equivalence class of nodes we get a differential $q_\infty$ on $X$ with a pair of half-infinite cylinders at each node.
(The differentials for each equivalence class should be produced using the same sequence $X_k$, so they will lie in the same linear space obtained as a limit of fibers of $QN$ over a subsequence of the $X_k$). With $q_\infty$ thus constructed in the boundary of $QN$, we can find a sequence $q_k\in QN$ converging to $q_\infty$.

Apply the Cylinder Deformation Theorem to shrink cylinders on $q_k$ without twisting them to obtain differentials $q_k' \in QN$ where all the cylinders under consideration have modulus bounded above and below. The Deligne-Mumford bordification is not locally compact, and we cannot yet assume that the $q_k'$ converge after passing to a subsequence. 

Let $S$ be the set of curves that are pinched on $X$. The Dehn space of $S$ is locally compact, so we can pass to a subsequence to assume that the images of the $q_k'$ in the Dehn space of $S$ converge; this is possible because all the $q_k'$ can be chosen to live on surfaces in a compact neighborhood around the image of $X$.

Using Corollary \ref{C:algebraic2}, we can find $g_k \in \Stab(\smash{\widehat{V}})\cap \Tw(S)$ so that $g_k q_k'$ converge to a differential $q$.  This $q$ has the desired properties. 
\end{proof}

\begin{corollary}\label{C:UnPinchOneAtATime}
Let $N\subseteq \cT_{g,n}$ be an algebraic totally geodesic submanifold and $X\in \ol{N}$. Then, for any N-equivalence class of nodes on $X$, there is a nearby point in $\ol{N}$ where these nodes are not pinched but all other nodes remain pinched.
\end{corollary}

\begin{proof}
This follows by applying cylinder deformations to the differential $q$ produced in the lemma. We can stretch all but one equivalence class out to become pairs of half-infinite cylinders, and stretch the remaining equivalence class to have large but finite modulus.  
\end{proof}

\subsection{Conclusion of the proof.} We can now complete the goal of this section. 

\begin{proof}[Proof of Theorem \ref{T:BoundaryTotallyGeodesic}]
Consider the closure of $N$ in the Deligne--Mumford bordification intersected with one of its strata where exactly $k$ classes of $N$-equivalent curves have been pinched to nodes. By Lemma \ref{L:CoverVariety}, this intersection covers a variety and is, in particular, a complex analytic subset. Let $L$ be an irreducible component of this set. 

Corollary \ref{C:ManyFiniteArea} gives that over every point of $L$ there is a linear subspace of the fiber of $\ol{QN}$ of dimension $d-k$ consisting of finite area differentials. Corollary \ref{C:NotJustContinuity} gives that this subspace is contained in $QL$.

We claim that $L$ has dimension at most $d-k$. This is true because Corollary \ref{C:UnPinchOneAtATime} guarantees it is contained at least $k$ levels deep in a stratification of the boundary. (We of course expect that $\dim L = d-k$ but the verification of this fact will be postponed to the next section).

Let $Q_{\rm alg}' L$ be the intersection of $\ol{QN}$ with the relevant part of the boundary (or rather, the union of the irreducible components of this intersection which lie over $L$). This satisfies all the properties required of $Q_{\rm alg} L$ except that we have not proven it is irreducible. By Corollary \ref{C:ManyFiniteArea}, each fiber of $Q_{\rm alg}' L$ is a union of linear subspaces, so in particular it is closed under the $\bC$ action which rescales quadratic differentials (equally on all components). We  obtain an associated complex analytic set $\smash{\bP Q_{\rm alg}' L}$  by projectivising fibers, and this still covers a variety. The map $\smash{\bP Q_{\rm alg}' L} \to L$ is proper, so the image of any irreducible component is an analytic subset \cite[page 290]{IntroCxGeom}. Hence irreducibility of $L$ implies that one of the irreducible components of $Q_{\rm alg}' L$ maps surjectively onto $L$, and we can define $Q_{\rm alg}L$ to be that component. 
\end{proof}

\section{The boundary is semisimple}

\subsection{Main definitions and statement.} The following are the two main definitions of this section.

\begin{definition}
Suppose $1\leq \ell < s$ and
$$L_1 \subseteq \mathcal{T}_{g_1,n_1}\times \cdots \times \mathcal{T}_{g_\ell,n_\ell}, \quad L_2 \subseteq \mathcal{T}_{g_{\ell+1},n_{\ell+1}}\times \cdots \times \mathcal{T}_{g_s,n_s}$$
are both subsets of products of Teichmüller spaces. Then we say that the set
 $$L := L_1 \times L_2 \subseteq \mathcal{T}_{g_1,n_1}\times \cdots \times \mathcal{T}_{g_s,n_s}$$
 is the product of $L_1$ and $L_2$. We also say $L$ is a product if this is true after permuting factors, and we use the natural extension of this definition allowing for more than two factors. 
\end{definition}

\begin{definition}
Following Definition \ref{D:TG}, suppose 
 $$L\subseteq \mathcal{T}_{g_1,n_1}\times \cdots \times \mathcal{T}_{g_s,n_s}$$
 is a totally geodesic complex analytic set. We say $L$ is simple if the projection to any factor is an isometric embedding. We say $L$ is semisimple if it is a product of simple totally geodesic complex analytic sets. 
\end{definition}

The main result of this section is the following.

\begin{theorem}\label{T:ss} 
Every algebraic totally geodesic complex analytic subset 
$$L\subseteq \mathcal{T}_{g_1,n_1}\times \cdots \times \mathcal{T}_{g_s,n_s}$$ 
is semisimple.
\end{theorem} 

Recall that, according to Definition \ref{D:TG}, the $L$ here is assumed to be irreducible. 

\subsection{Primality and the proof of Theorem \ref{T:ss}.}\label{SS:Primality}

Let $$L\subseteq \mathcal{T}_{g_1,n_1}\times \cdots \times \mathcal{T}_{g_s,n_s}$$
be a totally geodesic complex analytic set. Fix a $Q_{\rm alg}L \subseteq QL$ as given by Definition \ref{D:TG}.  We now set up notation to deal with an important technical detail.

Suppose $X\in L$, $V\subseteq QD(X)$ has dimension at least $\dim L$, and $V \subseteq Q_{\rm alg} L$. We assume $V$ is not contained in a larger subspace of the fiber of $Q_{\rm alg} L$ above $X$.

The components of $X$ are indexed by the numbers $\{1, \ldots, s\}$. Let $I_Z \subseteq \{1, \ldots, s\}$ correspond to the components of $X$ on which all differentials of $V$ are zero and $I_R$ correspond to the remaining components. Here ``$Z$'' stands for ``zero'' and ``$R$'' for ``rest''. Note that $\{1, \ldots, s\}$ is the disjoint union of $I_Z$ and $I_R$. Let 
$$Z : \prod_{i \in \{1, \ldots, s\}} \mathcal{T}_{g_i, n_i} \to \prod_{i \in I_Z} \mathcal{T}_{g_i, n_i}, \quad\quad R: \prod_{i \in \{1, \ldots, s\}} \mathcal{T}_{g_i, n_i} \to \prod_{i \in I_R} \mathcal{T}_{g_i, n_i}$$
be the associated maps that forget factors. 

Consider the slice of $L$ defined by 
$$ S = \{X'\in L : Z(X') = Z(X) \}.  $$
This is the subset of $L$ of all surfaces which are equal to $X$ in the coordinates where all differentials in $V$ are zero. We also consider the image $R(S)$, which we call the reduced slice, and note that $R : S \to R(S)$ is a bijection. 

Later we will see that all of the objects just introduced do not depend on the choice of $X$ and $V$, but until then we keep in mind the possibility that these objects, including $S$, may depend on these choices. 

Define $QS$ to be the set of pairs $(X', q') \in Q_{\rm alg} L$ with $X'\in S$ and $q'$ zero on all components indexed by $I_Z$. Let $R(QS)$ be the corresponding collection of quadratic differentials on elements of $R(S)$. Again, note that the corresponding forgetful map, which by abuse of notation we denote by $R : QS \to R(QS)$, is also a bijection.

We emphasize that no information is lost by $R$ in our context because we are restricting to points of $L$ where the components forgotten by $R$ are constant. Furthermore, we are restricting to differentials that are zero on the components that are forgotten. In our context, $R$ has an inverse which simply ``adds back" the $Z(X)$ that was forgotten. Our use of slices and the map $R$ is a technical detail required to connect to previous work considering products of strata of differentials, where by definition the differentials are non-zero on all components. 

Consider the typical stratum of a differential in $V$ and let $V^{(0)}\subseteq V$ denote the elements of $V$ in this stratum. Note that $V^{(0)}$ is open, connected, and dense in $V$. Consider the intersection of $R(QS)$ with  the associated stratum after applying $R$. Let $\cM$ be an irreducible component of this intersection that contains $R(V^{(0)})$. 

All the constructions so far are fully compatible with the maps
$\pi$ from products of Teichm\"uller spaces to the associated products of moduli spaces (of Riemann surfaces or quadratic differentials), so we get that $\cM$ covers a variety $\pi(\cM)$. By definition (and this is the point of our use of slices and forgetful maps) this $\pi(\cM)$ lies in a product of strata of \emph{non-zero} quadratic differentials. 

\begin{lemma}\label{L:MisGL2inv}
$\cM$ and $\pi(\cM)$ are $GL^+(2, \bR)$-invariant.
\end{lemma}

\begin{proof}
Recall from Definition \ref{D:TG} that $Q_{\rm alg} L$ is $GL^+(2, \bR)$-invariant. By definition of the action, it follows that $QS$ is invariant, and hence that $R(QS)$ is invariant as well. The action preserves strata, so it follows that the relevant intersection of $R(QS)$ with a stratum is invariant as well. The action on a stratum is real analytic, so it preserves irreducible components and we conclude that $\cM$ and $\pi(\cM)$ are $GL^+(2, \bR)$-invariant.  
\end{proof}

Recall from \cite{ChenWright} that an invariant subvariety of a product of strata of quadratic differentials is called prime if it cannot be written as a product. Furthermore, every invariant subvariety of a product of strata of quadratic differentials can be written uniquely as a product of primes. We will make use of the following  structure theorem; see \cite[Theorem 1.3]{ChenWright}. 

\begin{theorem}[Chen--Wright]
In any prime invariant subvariety of a product of strata, the absolute periods in one component locally determine the absolute periods in any other component. 
\end{theorem}

This is stated in \cite{ChenWright} in the case of Abelian differentials but it follows immediately that it also holds for quadratic differentials with the standard convention that period coordinates, and hence the notion of ``absolute periods", come from considering double covers where the quadratic differential pulls back to the square of an Abelian differential. 

It follows that we get a decomposition of $\pi(\cM)$ as a product of primes and hence a corresponding decomposition of $\cM$, which we denote 
$$\cM = \cM_1 \times \cdots \times \cM_k.$$

Recall that $\cM$ contains $R(V^{(0)})$, which is an open dense subset of $R(V)$. Since $V$ is not contained in a larger subspace, it follows that we get an associated decomposition of $V$, 
$$V= V_1 \oplus \cdots \oplus V_k, $$
where the $V_i$ can be defined as the differentials of $V$ supported on the same components as $\cM_i$. We now make the following key observation:

\begin{lemma}\label{L:ViOneZeroAllZero}
Suppose $q'\in V_i$ is zero on one of the components of $X$ corresponding to the factor $\cM_i$. Then $q'=0$. 
\end{lemma}
\begin{proof}
Suppose otherwise. Pick some arbitrary $q\in V^{(0)}$. Adding small multiples of $q'$ to $q$ gives rise to a path in $\cM_i$ were multiples of a fixed non-zero differential are added in some components and some other components remained unchanged. This causes area to change in some components but not others. (For example  one can consider large multiples instead of small to let a differential go to infinity and see that the area cannot be constant.) Since area can be computed from absolute periods we see that absolute periods change in some components but not all, contradicting the fact that $\cM_i$ is prime. 
\end{proof}

We now consider a non-standard exponential map 
$$\exp : V_1 \oplus \cdots \oplus V_k \to \mathcal{T}_{g_1,n_1}\times \cdots \times \mathcal{T}_{g_s,n_s}$$
defined as follows. Given $$v=(q_1, \ldots, q_k)\in V_1 \oplus \cdots \oplus V_k$$ let $\exp(v)$ be the result of applying geodesic flow for time $|q_i|$ using the differential $q_i$ for $i=1,\ldots, k$. Geodesic flows in different components commute, so it does not matter what order one applies these flows in or if one applies them simultaneously.

\begin{lemma}\label{L:expContinuous}
$\exp$  is continuous.
\end{lemma}

\begin{proof}
Consider a sequence $$v^{(j)}=(q_1^{(j)}, \ldots, q_k^{(j)}) \in V$$ converging to $v=(q_1, \ldots, q_k) \in V$. To show convergence of $\exp(v^{(j)})$ to $\exp(v)$, it suffices to prove convergence in each of the $k$ factors. Restricting to the factor $i$, we proceed in two cases. 

In the first case, suppose that $q_i\neq 0$. Then, Lemma \ref{L:ViOneZeroAllZero} gives that $q_i$ is non-zero in all components associated to this factor, and we get that eventually all the components of $\smash{q_i^{(j)}}$ are non-zero.  The result of flowing for time $\smash{|q_i^{(j)}|}$ along $\smash{q_i^{(j)}}$ converges to the result of flowing for time $|q_i|$ along $q_i$ because $\smash{|q_i^{(j)}|}\to |q_i|$ and the action is continuous when the set of components where the differential is non-zero is constant.

In the second case, suppose $q_i= 0$. Then $\smash{|q_i^{(j)}|}\to 0$ and continuity follows since flowing for amounts of time converging to zero results in no change in the limit. 
\end{proof}

\begin{lemma}\label{L:expInjective}
$\exp$  is injective.
\end{lemma}

\begin{proof}
Suppose $\exp(q_1, \ldots, q_k) = \exp(q_1', \ldots, q_k')$. The Teichm\"uller Uniqueness Theorem gives that $q_i$ and $q_i'$ are equal up to scale in each component. It follows  that $q_i=c_iq_i'$ for some $c_i>0$, since otherwise a linear combination of $q_i$ and $q_i'$ would contradict Lemma \ref{L:ViOneZeroAllZero}. The distance moved in each component of the $i$-th factor is by definition $|q_i| = |q_i'|$, so we get that $c_i=1$. 
\end{proof}

\begin{lemma}\label{L:expImageInS}
The image of $\exp$ is contained in the slice $S$.
\end{lemma}

\begin{proof}
Because $\cM$ is a product and is invariant under geodesic flow, it follows that $\cM$ is invariant under doing different amounts of geodesic flow in each factor $\cM_i$. By continuity and density, we get that same statement for the closure of $\cM$ in $R(QS)$, and hence also for the associated locus in $QS$. 
\end{proof}

\begin{lemma}
$\exp$ is a homeomorphism onto $S$ and $S=L$. Furthermore, $\dim V = \dim L$.
\end{lemma}

\begin{proof}
We know that $\exp$ is continuous and injective. Since it is also proper, the image of $\exp$ is a topological manifold of real dimension $2 \dim V$, where, as always, $\dim$ denotes complex dimension. Invariance of Domain and our assumption that $\dim V \geq \dim L$ gives $\dim V = \dim L$.
We will show that it follows from the irreducibility of $L$ that $L$ is equal to the image of this map. Since the image is contained in $S$ and $S\subseteq L$, this will also prove $S=L$. 

Indeed, $L$ can be stratified into real submanifolds (actually even complex submanifolds), and irreducibility of $L$ implies that there is only one largest dimensional stratum $L^{\mathrm{top}}$ in this stratification, which is furthermore open and dense in $L$ \cite[Section 18]{Whitney}. Let $V^{\mathrm{top}} \subseteq \oplus_{i=1}^k V_i$ 
be the preimage of $L^{\mathrm{top}}$ under $\exp$. Using Invariance of Domain as well as properness and injectivity of $\exp$, we see that the image of  $\exp|_{V^{\mathrm{top}}}$ is both open and closed in $L^{\mathrm{top}}$ and hence this image is equal to $L^{\mathrm{top}}$. Since the image of $\exp$ is closed and $L^{\mathrm{top}}$ is dense in $L$, we get that the image of $\exp$ is equal to $L$. 
\end{proof}

We get the following as as immediate corollary.

\begin{corollary}\label{C:DefineLi}
$L$ can be written as a product 
$$L=L_0 \times L_1 \times \cdots \times L_k$$
such that 
\begin{enumerate}
\item  $L_0$ is a subset of the product of Teichm\"uller spaces indexed by $I_Z$,
\item $L_0$ consists of a single point, and 
\item each $L_i, \ i>0$ is  a subset of the product of the factors corresponding to $\mathcal{M}_i$. 
\end{enumerate}
Furthermore, if $X=(X_0, X_1, \ldots, X_k)$ is the point chosen at the beginning of this subsection and $$Y=(Y_0, Y_1, \ldots, Y_k)\in L$$ is arbitrary, then $L_i$ contains a T-disc through $X_i$ and $Y_i$ for each $i$.
\end{corollary}

We can now conclude as follows:

\begin{lemma}
$I_Z, I_R$, the partition of $I_R$ into $k$ subsets associated to the product decomposition $\cM = \cM_1\times \cdots \cM_k$, and the $L_i$ do not depend on the choice of $X$ and $V$. 
\end{lemma}

\begin{proof}
We see from the above that $I_Z$ indexes the components where $L$ is constant, so $I_Z$ and its complement $I_R$ do not depend on our choices. We now consider the partition of $I_R$ into $k$ subsets. By Remark \ref{R:Tdiscs} part \eqref{R:Tdiscs:equal}, $L_i$ cannot be a product in a non-trivial way, and it follows that the decomposition of $I_R$ and the $L_i$ cannot depend on our choices either. 
\end{proof} 

\begin{proof}[Proof of Theorem \ref{T:ss}]
Since $X$ was arbitrary and the $L_i$ do not depend on choices, we get that the $L_i$ contain a T-disc through every pair of distinct points. Thus Remark \ref{R:Tdiscs} part \eqref{R:Tdiscs:equal} gives the result. A technical point is that if a complex analytic set is a product of sets, then the factors are automatically complex analytic. (For example, if $A'\times B'$ is complex analytic in a complex manifold $A\times B$, note that $A'$ can be expressed in the form $(A'\times B') \cap (A \times \{b_0\})$ for some $b_0 
\in B'$, and use that the intersection of analytic sets is analytic).
\end{proof}

\section{The boundary is smooth and irreducible}

\subsection{Smoothness and the proof of Theorem \ref{T:main1}.} The following ties up a loose end in previous discussions.

\begin{lemma}\label{L:smooth}
Every totally geodesic complex analytic set $L$ of a product of Teichmüller spaces is a complex submanifold. 
\end{lemma}

Note that the arguments above show that $L$ is a topological manifold, but these arguments do not show it is a complex submanifold, and the exponential maps lack the smoothness to be immediately useful for this problem \cite[Theorem 1]{Rees}.

\begin{proof}
By Theorem \ref{T:ss} it suffices to prove this when $L$ is simple. For this it is helpful to use the version of the exponential map defined on the tangent space, so $\exp_X(v)$ is the constant speed (but typically not unit speed) geodesic through $X$ with derivative $v$. By first considering the smooth points of $L$, as above we see that most and hence all points $X$ of $L$ have a $\dim L$-dimensional subspace of the ambient tangent space whose image under $\exp_X$ is equal to $L$. This shows that $L$ is locally well approximated by a linear space at every point of $L$, as is the case for a submanifold. The result now follows from \cite[Proposition 3.4]{LipReg}, keeping in mind that all Finsler metrics are equivalent on compact sets. 

(The idea of the argument is to consider (locally) the projection to the approximating linear space. If this projection has degree 1, one can find a holomorphic inverse off of the branching locus, and then extend it to the branching locus using the Removable Singularity Theorem. This would show that $L$ is locally the graph of a holomorphic map, giving the result. On the other hand, if the degree is greater than 1, then as points in the same fiber come together one find pairs of points such that the most efficient way to join them is to leave $L$). 
\end{proof}

\subsection{Irreducibility.} The final loose end is the following, which will be important later. 

\begin{proposition}\label{P:irreducible}
Suppose $N\subseteq \mathcal{T}_{g,n}$ is algebraic and totally geodesic and let $L$ denote the intersection of $\overline{N}$ with a boundary stratum. Then $L$ is irreducible. 
\end{proposition}

We will prove some lemmas before we prove this proposition. 

\begin{lemma}\label{L:LimitOfDiscs}
Suppose $X_k, Y_k \in \cT_{g,n}$, with $X_k\to X$ and $Y_k\to X$ in the Deligne-Mumford bordification. Suppose that $X,Y$ are in the same stratum, that $X_k\neq Y_k$ for all $k \in \mathbb{N}$, and that $\sup_k d(X_k, Y_k)<\infty$. Let $\Delta_k$ be the Teichm\"uller disc through $X_k$ and $Y_k$. Then, possibly after passing to a subsequence, the discs $\Delta_k$ converge to a holomorphic disc in the stratum of $X$ and $Y$ containing $X$ and $Y$. 
\end{lemma}

\begin{proof}
For the rest of the proof let $\rho$ denote the map that forgets a differential on a surface, so $\rho(X,q)=X$. 

Let $q_k$ be a quadratic differential on $X_k$ and $t_k>0$ be a real number such that doing geodesic flow along $q_k$ for time $t_k$ starting at $X_k$ gives $Y_k$. Pass to a subsequence so $t_k\to t$. Also pass to a subsequence to assume there is a meromorphic differential $q$ on $X$ that is non-zero on every component  and such that, with appropriate rescaling depending on the component, $q_k$ converges to $q$ (away from nodes).  Compare for example to \cite[Theorem 10]{SumOfExponents}  and the references therein.

For any matrix $g\in GL(2, \bR)$ one can check directly using the quasi-conformal topology that  $\rho(g q_k)$ converges to $\rho(gq)$, and, in particular, that $\rho(g_tq) =Y$. Let $\Delta$ be the image under $\rho$ of the $GL^+(2,\bR)$ orbit of $q$. It follows that the $\Delta_k$ converge to $\Delta$. One can check that $\Delta$ is holomorphic using holomorphic dependence of parameters in the measurable Riemann mapping theorem. 

(An alternative approach to this lemma would be to use normal families and the fact that Dehn space is a bounded domain \cite[Section 16]{BersDehnSpaceBounded}). 
\end{proof}

\begin{remark}
One can compare the previous proof to \cite[Theorem 1.4]{MultiScale}, keeping in mind that Lemma \ref{L:LimitOfDiscs} is much simpler because it does not require any global structural results for a compactification. 
\end{remark}

\begin{lemma}\label{L:CarefulOpeningUp}
In the setting of Proposition \ref{P:irreducible}, given two points $X, Y\in L$, it is possible to find $X_k, Y_k \in N$ such that $X_k\to X$, $Y_k\to Y$, and the distance between $X_k$ and $Y_k$ is bounded. 
\end{lemma}

\begin{proof}
Lemma \ref{L:DilationLemma} gives  differentials in $Q N$ such that replacing some disjoint equivalence classes of cylinders with pairs of half-infinite cylinders gives differentials on $X$ and $Y$. We can define $X_k$ and $Y_k$ to be the result of using the Cylinder Deformation Theorem to stretch cylinders (without twisting them) so that a chosen cylinder in each horizontal equivalence class has modulus $k$. Then all the moduli of the relevant cylinders are the same at $X_k$ and $Y_k$. 

One can then directly build quasi-conformal maps from $X_k$ to $Y_k$ of bounded dilatation which are the identity on a large part of the horizontal cylinders.
\end{proof}

\begin{remark}
For convenience we gave a proof of Lemma \ref{L:CarefulOpeningUp} that is very specific to our situation. 
There are other approaches that hinge on finding holomorphic disks  through the given points in the variety and using plumbing coordinates. One such approach might try to bound distances using the Minsky product region estimates, leaning on asymptotic results on the hyperbolic metric to control lengths and twists for the short curves: see for example \cite[page 71]{Wol10} and \cite{WolpertUniversalCurve}. Another approach would be to estimate the Teichm\"uller norm of an infinitesimal change of the plumbing parameter and (roughly speaking) integrate that to get bounded distance.  Another approach would be to do something  similar to what we did above but without the use of flat geometry. We thank Scott Wolpert for helpful discussions about several of these approaches. 
\end{remark}

\begin{proof}[Proof of Proposition \ref{P:irreducible}]
Lemma \ref{L:LimitOfDiscs} applied to the points produced by Lemma \ref{L:CarefulOpeningUp} gives that every pair of points in $L$ is contained in a holomorphic disc in $L$, which shows that $L$ is irreducible. 
\end{proof}

\subsection{Proof of Theorem \ref{T:main1}} This concludes the proof of Theorem \ref{T:main1} as follows:

\begin{proof}[Proof of Theorem \ref{T:main1}]
Irreducibility is given by Proposition \ref{P:irreducible}. Algebraicity follows from Lemma \ref{L:CoverVariety}, and semisimplicity from Theorem \ref{T:ss}, with the additional input of Lemma \ref{L:smooth} guaranteeing that the complex analytic sets that appear are in fact complex submanifolds. 
\end{proof}

\section{First consequences of semisimplicity}

In this section we begin the work of clarifying how some of the different structures previously introduced relate to each other.

\subsection{The boundary of a subsurface.} If $U$ is a subsurface, we let $\cT(U)$ denote the Teichm\"uller space of the finite type surface obtained by contracting each boundary component of $U$ to a marked point or puncture.

\begin{lemma}\label{L:BoundaryCurvesSimpleFactor}
Let $N \subseteq \mathcal{T}_{g,n}$ be an algebraic totally geodesic submanifold of Teichm\"uller space. Suppose  $F \subseteq \cT(U_1)\times \cdots \times \cT(U_k)$
is a simple factor  of the intersection of $\ol{N}$ with a stratum of the Deligne-Mumford bordification. Furthermore, assume $F$ is not a point. Then:
\begin{enumerate} 
\item\label{L:BoundaryCurvesSimpleFactor:1}  For all $i\neq j$, every boundary curve of $U_i$ is equivalent to a boundary curve of $U_j$.
\item\label{L:BoundaryCurvesSimpleFactor:2} For all cylinder curves in $\cC(U_i)$, every equivalent curve lies in one of the $\cC(U_j)$, and there is at least one equivalent curve in each $\cC(U_j)$.
\item\label{L:BoundaryCurvesSimpleFactor:3} The boundary of the union of the $U_i$ is a union of equivalence classes. 
\end{enumerate} 
\end{lemma}

Here, saying a curve is in $\cC(U_i)$ is just a convenient way of saying it is contained in $U_i$ non-peripherally. A simple factor is by definition one of the $L_i, \ i\geq0$ in Corollary \ref{C:DefineLi}; the requirement that $F$ not be a point exactly excludes $L_0$.

Note that assertion \eqref{L:BoundaryCurvesSimpleFactor:2} is trivially true when $F$ is a point, but that assertion \eqref{L:BoundaryCurvesSimpleFactor:1} fails, for example, when $N=\cT_{g,n}$ and a convenient multi-curve is pinched. In this case, when $F$ is a point, the $U_i$ are the pants components of the complement of the pinched multi-curve. See Figure \ref{F:U1U2}.  

\begin{figure}[h!]
\includegraphics[width=0.4\linewidth]{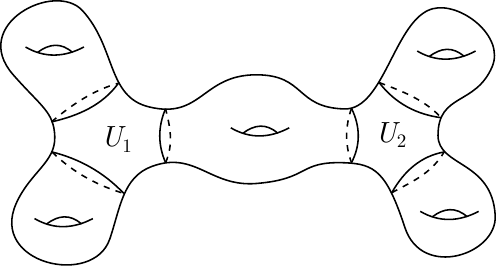}
\caption{If $N$ is the whole Teichm\"uller space, pinching the curves drawn gives a simple factor which is a point in $\cT(U_1)\times \cT(U_2)$ for the trivial reason that $\cT(U_1)$ and $\cT(U_2)$ are points themselves. In this example each curve is only equivalent to itself.}
\label{F:U1U2}
\end{figure}  

\begin{proof}
Suppose to the contrary that there is a boundary curve $\alpha$ of $U_1$ that is not equivalent to any boundary curve of $U_2$. By Corollary \ref{C:UnPinchOneAtATime}, we can unpinch $\alpha$ and its equivalence class, without unpinching any other curves. We get a new boundary stratum where $U_2$ is still a component and where there is a component $U_1'$ containing $U_1$. Note that $U_1'$ and $U_2$ must be in the same simple factor, since otherwise we can take limits and contradict the fact that $F$ is a simple factor. However, a Dehn multi-twist along the  equivalence class of $\alpha$ (produced by Lemma \ref{L:InvariantUnderTwist}) changes $U_1'$ without changing $U_2$, giving a contradiction. A similar argument proves the second claim after pinching the boundary curves of the $U_i$. 

It remains only to address the third claim, where we will require some ideas from papers on invariant subvarieties. These ideas are isolated in the second of the following two sublemmas, so they can be treated as a black box if desired. 

\begin{sublemma}\label{SL1} 
Fix an equivalence class $B$ of a boundary curve of one of the $U_i$. There is a finite area quadratic differential $q$ in the boundary of $Q_{\max} N$ which is horizontally periodic and such that the following hold: 
\begin{enumerate}
\item the curves in $B$ are horizontal cylinders on $q$,
\item there is precisely one other equivalence class $E$ of horizontal cylinders on $q$,  
\item the pinched curves at $q$ are exactly the other equivalence classes of boundary curves, and
\item all core curves of cylinders in $E$ are contained non-peripherally in one of the $U_i$. 
\end{enumerate}
\end{sublemma} 

\begin{proof}[Proof of Sublemma \ref{SL1}]
Because $F$ is not a point, it contains non constant T-discs. The quadratic differentials generating these T-discs have cylinders, and so we see that the $U_i$ contain cylinder curves non-peripherally. Consider such a cylinder curve and note that the second claim of Lemma \ref{L:BoundaryCurvesSimpleFactor} gives that its equivalence class $E$ consists entirely of curves that that are contained non-peripherally in one of the $U_i$. 

Pinch all the equivalence classes of boundary curves of the $U_i$ except $B$. Note that the curves of $B$ and $E$ must all lie in the same simple factor, again by the second claim of Lemma \ref{L:BoundaryCurvesSimpleFactor}. 

Lemma \ref{L:InvariantUnderTwist} gives a positive Dehn multi-twist in all the  curves of $E \cup F$ that stabilizes the given simple factor. Consider the T-disc in that simple factor that contains a point and its image under this multi-twist, and consider the quadratic differential $q$ generating this T-disc which also generates the horocycle between these two points in this disc. By applying Lemma \ref{L:StrebelGeneratingHorocycle} in each component, we see that it is horizontally periodic with horizontal cylinders corresponding to $B\cup E$. 
\end{proof}

\begin{sublemma}\label{SL2}
Let  $\mathbf{C}_1$ and $\mathbf{C}_2$ be parallel equivalence classes of cylinders on a differential $q$ in the boundary of $Q_{\max} N$. If a cylinder of  $\mathbf{C}_1$ is adjacent to a cylinder of $\mathbf{C}_2$, then every cylinder of $\mathbf{C}_2$ is adjacent to a cylinder of $\mathbf{C}_1$. 
\end{sublemma}

Here we say two cylinders are adjacent exactly when they have a common saddle connection on their boundary. 

\begin{proof}[Proof of Sublemma \ref{SL2}]
This follows from an ``over-collapsing" technique developed for the classification of invariant subvarieties of differentials. See \cite[Lemma 4.4]{ApisaMHD} for a precise result sufficient for our needs and see the paragraph after \cite[Lemma 4.8]{AWgeneralizations} for a list of similar arguments which have appeared in the literature, many of which include illustrations. The rough summary of the argument is that one considers an appropriate linear path of cylinder deformations of $\mathbf{C}_1$  which keeps the circumferences the same but decreases the areas of these cylinders to 0. Continuing this path in the stratum, one finds differentials where the moduli of some of the cylinders adjacent to $\mathbf{C}_1$ start to decrease. Since this ratio of moduli of cylinders in $\mathbf{C}_2$ is constant by Lemma \ref{L:ConstRatio}, this gives the result. 
\end{proof} 

To conclude the proof of the third claim, consider the differential $q$ given by Sublemma \ref{SL1}. Let $\mathbf{C}_1=E$, and let $\mathbf{C}_2=B$ denote the equivalence class of one of the boundary curves of one of the $U_i$. Then Sublemma \ref{SL2} gives the result because the curves of $E$ are contained in the $U_i$. 
\end{proof}

\subsection{Projection to a subsurface.}  We also note the following result, which will clarify much of our later discussion. 

\begin{lemma}\label{L:SimpleFactorsWellDefined}
Suppose  $F \subseteq \cT(U_1)\times \cdots \times \cT(U_k)$
is a simple factor  of the intersection of $\ol{N}$ with a stratum of the Deligne-Mumford bordification.  Fix $1\leq i_0\leq k$. Consider any non-empty intersection of $\ol{N}$ with a stratum of the Deligne-Mumford bordification which has $\mathcal{T}(U_{i_0})$ as a factor. 
\begin{enumerate}
\item Suppose $F$ is not a point. Then this intersection has $F$ as a factor. 
\item Suppose $F$ is a point. Then the projection of this intersection to $\mathcal{T}(U_{i_0})$ is the same as the projection of $F$ to $\mathcal{T}(U_{i_0})$, and so, in particular, a point. 
\end{enumerate}
\end{lemma}

\begin{proof}
In either case, we can consider the component of the boundary of $\ol{N}$ resulting from pinching the equivalence classes of the boundary curves of $U_{i_0}$. Let $F'$ be the simple factor involving the $\mathcal{T}(U_{i_0})$ coordinate.  

Any boundary component of $\ol{N}$ where  $U_{i_0}$ appears results from pinching additional equivalence classes, with the restriction that the equivalence classes must contain disjoint curves and cannot contain curves in $U_{i_0}$. 

If $F'$ is a point, then any additional degeneration will not affect the projection to $\mathcal{T}(U_{i_0})$. If $F'$ is not a point, then the product structure of the boundary gives that any additional degenerations which leaves $U_{i_0}$ intact cannot affect $F'$. 
\end{proof}

\begin{remark}
In Figure \ref{F:U1U2} we can first pinch the curves bounding $U_1$ and then pinch the curves bounding $U_2$ to see that a simple factor which is a point can in a certain sense ``grow'' from involving only $U_1$ to involving  both $U_1$ and $U_2$. 
\end{remark}

Lemma \ref{L:SimpleFactorsWellDefined} says that, with a minor caveat in the case of simple factors that are points, the factor $F$ is well defined given any of the $U_i$.

\section{From Teichm\"uller spaces to non-annular curve graphs}

Theorem \ref{T:main1} provides a decomposition of the boundary of any algebraic totally geodesic submanifold into simple factors. The projection of any such factor into any of its Teichmüller coordinate domains is an isometric embedding whose image is an algebraic totally geodesic submanifold. The main results of this section is the following theorem which provides control over the coarse geometry of these projections and which is particularly useful for the applications we discuss later. Recall that if $U$ is a surface then $\mathcal{T}(U)$ and $\mathcal{C}(U)$ denote its Teichmüller space and curve complex, respectively.

\begin{theorem}\label{T:CurveSS}
Let $N \subseteq \mathcal{T}_{g,n}$ be an algebraic totally geodesic submanifold. Then, there exists $K > 0$ such that if
$$F \subseteq \cT(U_1)\times \cdots \times \cT(U_k)$$
is a simple factor of the intersection of $\ol{N}$ with a stratum of the Deligne--Mumford bordification, then the projection maps from the image of $F$ in 
$$\cC(U_1)\times \cdots \times \cC(U_k)$$
to each factor are $K$-quasi-isometric embeddings. 
\end{theorem}

As a corollary, it is possible to deduce the same statement without the need to pass to the boundary. 

\begin{corollary}\label{C:CurveSS}
Let $N \subseteq \mathcal{T}_{g,n}$ be an algebraic totally geodesic submanifold. Then, there exists $K > 0$ such that if the subsurfaces $U_1, \ldots, U_k$ arise as in Theorem \ref{T:CurveSS} from a simple factor of the boundary, then the projection maps from the image of $N$ in 
$$\cC(U_1)\times \cdots \times \cC(U_k)$$
to each factor are $K$-quasi-isometric embeddings. 
\end{corollary}

We isolate some of the work of proving the corollary in the following lemma.

\begin{lemma}\label{L:OpenUpNoProjChange}
Let $N \subseteq \mathcal{T}_{g,n}$ be an algebraic totally geodesic submanifold.
There exists a (small) $\epsilon>0$ and a (large) $K>0$ such that the following holds for any $Y\in N$. 

For any collection of equivalence classes of curves, all of hyperbolic length at most $\epsilon$ at $Y$,  there is a point $Z$ of the boundary of $N$ where exactly these curves are pinched, and for every subsurface $U$ that is a component of the complement of the union of the curves, the projections of $Y$ and $Z$ to $\cC(U)$ are within $K$ of each other. 
\end{lemma}

\begin{proof}[Proof sketch]
We will use compactness of the boundary of $\pi(N)$ in $\ol{\cM}_{g,n}$. Each point $Z$ in the boundary has a small neighborhood consisting of points $Y$ which are appropriately close to $Z$. Taking a finite subcover and lifting gives the result. (One should use neighborhoods that have nice preimages in the bordification.) 
\end{proof}

\begin{proof}[Proof of Corollary \ref{C:CurveSS} assuming Theorem \ref{T:CurveSS}]
By Corollary \ref{C:algebraic1}, there are only finitely many orbits of relevant tuples $U_1, \ldots, U_k$, so we do not have to worry about the dependence of $K$ on this tuple. 

It suffices to show that for every point $X$ in $N$ there is a corresponding point of $F$ with coarsely the same image in the given product of curve graphs. We first do this when $F$ is not a point. 

Let $C$ be the constant from Lemma \ref{L:ratio} controlling ratios of hyperbolic lengths, and let $\epsilon$ be as in Lemma \ref{L:OpenUpNoProjChange}. 
We will say that a curve is short if its hyperbolic length is at most $\epsilon$ and very short if its hyperbolic length is at most $\epsilon/C$. Thus Lemma \ref{L:ratio}  gives that if a curve is very short, all equivalent curves are short. 

Let $X$ be a point of $N$. Consider a Teichm\"uller geodesic segment from $X$ to a point $Z$ of $N$ where all boundary curves of $U_1,\dots,U_k$ are very short. Let $Y$ be the first point on this segment such that there exists $i$ such that all boundary curves of $U_i$ are very short at $Y$. By definition, at every point on the geodesic before $Y$, each $U_i$ has a boundary curve with length greater than $\epsilon/C$. 

By Rafi's theory of active intervals \cite[Section 3]{RafiComb}, which we recall in Theorem \ref{T:Active}, the projection to the curve complexes $\mathcal{C}(U_i)$ coarsely does not change along the Teichmüller geodesic $[X, Y]$. 

By Lemma \ref{L:BoundaryCurvesSimpleFactor} part \eqref{L:BoundaryCurvesSimpleFactor:1}, all the subsurfaces $U_i$ are bounded by curves in the same finite collection of equivalence classes. As discussed above, it follows that all boundary curves of all $U_i$ are short at $Y$.  We now let $Z$ be the point produced from $Y$ by Lemma \ref{L:OpenUpNoProjChange} with the collection of curves being the boundary of the $U_i$. So the curves pinched at $Z$ are exactly the boundary curves of the $U_i$. Lemma \ref{L:OpenUpNoProjChange} gives that $Y$ and $Z$ have coarsely the same projection to each $\cC(U_i)$. 

This shows that the image of $X$ in $\cC(U_1)\times \cdots \times \cC(U_k)$ is coarsely equal to the image of a point $Z$ in the stratum of the Deligne--Mumford bordification where the boundary components of all the $U_i$ have been pinched. This part of the boundary is irreducible by Proposition \ref{P:irreducible}. The given $F$ may have originated from an even smaller boundary stratum, i.e., with more pinched curves, but this is not a concern thanks to Lemma \ref{L:SimpleFactorsWellDefined}.

If $F$ is a point, we cannot use Lemma \ref{L:BoundaryCurvesSimpleFactor}, but the proof can be easily modified. Say $U_{i_0}$ has that all boundary curves are very short at $Y$. They key observation is that subsurface projection to $U_{i_0}$ cannot change along $[Y,Z]$, because by Lemmas \ref{L:OpenUpNoProjChange} and  \ref{L:SimpleFactorsWellDefined} it is controlled by  a single point of $\cT(U_{i_0})$. 
\end{proof}

\begin{proof}[Proof of Theorem \ref{T:CurveSS}]
Notice that Theorem \ref{T:main1} gives the corresponding result at the level of Teichm\"uller spaces. Lemma \ref{L:BoundaryCurvesSimpleFactor} part \eqref{L:BoundaryCurvesSimpleFactor:2} gives that for any cylinder curve in $\cC(U_i)$ there must be at least one equivalent cylinder curve in  every other $\cC(U_j)$.  

Hence Corollary \ref{C:Pinchable} and Lemma \ref{L:ratio} give that when a curve is short on one $U_i$ then there is a comparably short curve in the same equivalence class on every other $U_j$. Hence the result follows from Theorem \ref{T:electrification} and Corollary \ref{C:electrification}, because we can use Theorem  \ref{T:electrification} to coarsely compute distance in the curve graph of one component using distance in an electrification, and Corollary \ref{C:electrification} gives that the exact same electrification coarsely computes distance in the curve graphs of the other components. 
\end{proof}

\section{From Teichm\"uller spaces to annular curve graphs}

The goal of this section is to show that the analogue of Corollary \ref{C:CurveSS} for annular curve graphs also holds. 
See Section \ref{SS:AnnularSetup} for background. If $\alpha$ is a simple closed curve we let $\cC(\alpha)$ denote the Teichm\"uller space version of the annular curve graph of $\alpha$, i.e., $\cC(\alpha)$ can be viewed as the horoball in the upper half plane $\bH$ with imaginary part at least 1. There is a projection map $\pi_\alpha$ from Teichm\"uller space to $\cC(\alpha)$, where the imaginary part is the max of 1 and 1 over the hyperbolic length of $\alpha$, and the real part encodes twisting, as measured by taking an appropriate marking and projecting it to the mapping class group version of the annular curve graph. 

\begin{theorem}\label{T:AnnularCurveSS}
Let $N \subseteq \mathcal{T}_{g,n}$ be an algebraic totally geodesic submanifold. Then, there exists $K > 0$ such that if $\{\alpha_1, \ldots, \alpha_k\}$ is an $N$-equivalence class of simple closed curves, then the projection maps from the image of $N$ in 
$$\cC(\alpha_1)\times \cdots \times \cC(\alpha_k)$$
to each factor are $K$-quasi-isometric embeddings. 
\end{theorem}

We verify this using the following. 

\begin{lemma}\label{L:ForAnnularCurveSS}
Let $\cX$ be a subset of $\{z: \Im(z)>1\}^k \subset \bH^k$. Assume there are constants $C>0$ and $c_{j,m}> 0$ such that 
if $$(x_j+i y_j)_{j=1}^k, (x_j'+i y_j')_{j=1}^k\in \cX$$ then for all $j,m$ we have
$y_j/y_m < C$ and $$\left| |x_j-x_j'|-c_{j,m}|x_m-x_m'|\right| \leq  C(|y_1|+|y_1'|).$$
Then there is a $K$ depending only on $C, c_{j,m}$ and $k$ such that each of the $k$-coordinate projections $\cX\to \bH$ are $K$-quasi-isometric embeddings. 
\end{lemma}

The metric on $\bH^k$ here is the sup of the usual hyperbolic metrics on the factors. (Since everything here is coarse it wouldn't matter if we instead used the product Riemannian metric.) 

\begin{proof}
The distance between two points $x+iy, x'+iy'\in \bH$  is coarsely
$$\log(\max(y, y')/\min(y,y')) + \max(0,\log(|x-x'|/\max(y, y'))),$$
and the result follows from this. 
%
%
%
\end{proof}

\begin{proof}[Proof of Theorem \ref{T:AnnularCurveSS}]
We verify the assumptions of Lemma \ref{L:ForAnnularCurveSS}. 
Lemma \ref{L:ratio} immediately gives the assumption on ratios of $y$ coordinates. 

Consider two points $X, X'\in N$. Let $D_\alpha$ denote a multi-twist supported on the set of $\alpha_j$, as produced by Lemma \ref{L:InvariantUnderTwist}. We can pick a power $p$ such that the projections $\pi_{\alpha_1}(D_\alpha^pX')$ and $\pi_{\alpha_1}(X)$ have that their real parts, i.e., their twists, are uniformly close depending only on $N$. (The uniformity here follows from Corollary \ref{C:algebraic1}, which gives that there are only finitely many equivalence classes up to the action of the fundamental group of $N$. See also \cite[Theorem 5.1]{MW17} and \cite[Theorem 5.3]{ChenWright} for previous justifications of this.) 

It now suffices to prove that for every other $j$ the real parts of $\pi_{\alpha_j}(D_\alpha^pX')$ and $\pi_{\alpha_j}(X)$ must be equal up to an error of size $O(\ell_{\alpha_j}(X)^{-1}+\ell_{\alpha_j}(X')^{-1})$. (The constants $c_{j,m}$ arise from the ratio of twisting done by $D_\alpha$ in different curves.)  

Suppose otherwise in order to find a contradiction. We consider the geodesic segment from $X$ to $D_\alpha^pX'$ and perform a more in depth version of the analysis of Lemma \ref{L:RafiCyl}. Although we do not use the main results in Appendix \ref{A:Rafi}, we use results of Rafi recalled there, so we recommend first reading the beginning of that appendix. One should also recall that \cite[Theorem 4.3]{RafiComb} shows that
different notions of twisting at $\alpha$ (namely the $x$ coordinate of $\pi_\alpha$ and the $\twist_q$ of the appendix) are equivalent up to an error of size at most $O(\ell_{\alpha}(X)^{-1})$.

Let $q$ be the quadratic differential on $X$ generating the geodesic segment from $X$ to $D_\alpha^pX'$. Without loss of generality suppose in order to find a contradiction that the real parts of $\pi_{\alpha_2}(D_\alpha^pX')$ and $\pi_{\alpha_2}(X)$ differ by an extremely large multiple of $\ell_{\alpha_j}(X)^{-1}+\ell_{\alpha_j}(X')^{-1}$. As in the proof of Lemma \ref{L:RafiCyl}, we see that $\alpha_2$ is a cylinder on $q$. Moreover, the formulas recalled in the appendix give that, as always up to the error of one over hyperbolic length, the amount of twisting accomplished depends on the modulus. Thus we see that, at some point on the geodesic segment, the $\alpha_2$ cylinder must have large modulus, but that the $\alpha_1$ cylinder cannot have large modulus because almost no twisting is accomplished in $\alpha_1$; here we are using that the balanced times, in Rafi's terminology, for $\alpha_1$ and $\alpha_2$ are the same because the corresponding cylinders are parallel. This contradicts Corollary \ref{C:Connected} which gives the moduli of the $\alpha_1$ and $\alpha_2$ cylinders must have fixed ratio. 
\end{proof}

\section{Hierarchical hyperbolicity}

\subsection{The axioms.} We note in advance that to improve readability for non-experts  we will verify a slight modification of the axioms for an HHS. Indeed, we will verify the axioms for an almost HHS. These axioms are identical to those of an HHS, except for a technical change to the orthogonality axiom. 
It is known that every almost HHS is an HHS, see \cite[Appendix]{AlmostHHS}, so the very slightly modified axioms are equivalent to the usual ones. Additionally and for the sake of completeness, we sketch a proof of the usual orthogonality axiom at the end of this section.

We we will paraphrase the axioms for an almost HHS imprecisely, omitting a number of details; the reader should consult other sources, such as \cite{AlmostHHS}, for the precise definitions of these axioms. Our goal here is to briefly introduce the axioms for readers new to hierarchical hyperbolicity and to refresh the memory of those that are already familiar with this concept.

Roughly speaking, a quasi-geodesic metric space $X$, endowed with a good deal of extra structure specified below, is said to be hierarchically hyperbolic if the following axioms hold with uniform constants:

\begin{enumerate}
\setcounter{enumi}{-1}
\setlength{\itemsep}{4pt}
\item\label{O:MetricSpaces} \textbf{Basic set up.} There is an index set $\mfS$, called the set of domains, and for each $W\in \mfS$ there is a Gromov hyperbolic metric space $\cC W$. 
\item\label{O:Projections} \textbf{Projections.} There are  coarsely Lipschitz projections $\pi_W: X \to 2^{\cC W}.$ The image $\pi_W(x)$ of each point must have uniformly bounded diameter, so each $\pi_W(x)$ can be considered as a coarse point and there is no harm in thinking of $\pi_W$ as a map to $\cC W$ instead of to subsets of $\cC W$. The image of each $\pi_W$ is required to be quasi-convex (which is automatic when it is coarsely surjective). 
\item\label{O:Nesting} \textbf{Nesting.} The set of domains $\mfS$ has a reflexive partial order $\sqsubseteq$ called nesting with a unique maximal element. If $V \sqsubsetneq W$, then there is a specified point $\rho^V_W \in \cC W$ and a map $\rho^W_V: \cC W \to 2^{\cC V}.$
\item\label{O:Orthogonality} \textbf{Orthogonality.} The set of domains $\mfS$ has a symmetric, anti-reflexive relation $\perp$ called orthogonality, such that the following hold: 
\begin{enumerate}
\item\label{O:NestedInOrthogonal} If $V \sqsubseteq W$ and $W\perp U$ then $V \perp U$. 
\item\label{O:NotNestedAndOrthogonal} A nested pair is never orthogonal. 
\item\label{O:BoundedOrth} The cardinality of any set of pairwise orthogonal domains is uniformly bounded above.  
\end{enumerate}
\item\label{O:TC} \textbf{Transversality and consistency.} We say two domains $V$ and $W$ are transverse and write $V\pitchfork W$ if they are not orthogonal and neither is nested in the other. Furthermore, for a uniform constant $\kappa_0 > 0$:
\begin{enumerate}
\item\label{O:TC:Behrstock}  If $V\pitchfork W$, then there are points $\rho^V_W \in \cC W$ and $\rho^W_V \in \cC V$ such that the following Behrstock inequality holds:
$$\min\{d_W(\pi_W(x), \rho^V_W), d_V(\pi_V(x), \rho^W_V)\} \leq \kappa_0.$$
\item\label{O:TC:Functoriality} If $V\sqsubseteq W$, either $\pi_V(x)$ is approximately $\rho^W_V(\pi_W(x))$ or 
$$d_W(\pi_W(x), \rho^V_W)  \leq \kappa_0.$$ 
\item\label{O:TC:Rhos} If $U \sqsubseteq V$, then we have 
$$d_W(\rho^U_W, \rho^V_W) \leq \kappa_0$$ 
for all $W$ satisfying one of the following: $V\sqsubseteq W$; or $W$ transverse to $V$ and not orthogonal to $U$.
\end{enumerate}
\item\label{O:Finite} \textbf{Finite complexity.} There is a uniform upper bound on the length of properly nested chains. 
\item\label{L:LargeLinks} \textbf{Large links.} There are uniform constants $\lambda > 0$ and $E > 0$ such that for every $x, x'\in X$ there exists a set $\{T_i\}$ of at most $$\lambda d_W(\pi_W(x), \pi_W(x')) + \lambda$$ domains properly nested in $W$ such that if $T$ is properly nested in $W$ and 
$$d_T(\pi_T(x), \pi_T(x'))>E$$
then $T$ is nested in one of the $T_i$. 
\item\label{O:BGI} \textbf{Bounded geodesic image.} If $\gamma$ is a geodesic in $\cC W$ and $V \sqsubseteq W$, then either $\rho^W_V(\gamma)$ has uniformly bounded diameter, or $\gamma$ comes close to $\rho^V_W$. 
\item\label{O:PR} \textbf{Partial realization.} For a uniform $\alpha > 0$ and any pairwise orthogonal family $\{V_j\}$ and points $p_j \in \pi_{V_j}(X)$, there exists a point $x\in X$ such that
$$d_{V_j}(\pi_{V_j}(x), p_j)\leq \alpha$$
for all $j$, and such that for each $j$ and each $V$ with  $V_j \sqsubsetneq V$ or $V_j\pitchfork V$, 
$$d_V(\pi_V(x), \rho^{V_j}_V)\leq \alpha.$$ 
\item\label{O:Uniqueness} \textbf{Uniqueness.} For all $\kappa > 0$ there exists $\theta_u > 0$ such that if $d(x,y)\geq \theta_u$ then there exists $V$ with $d_V(\pi_V(x), \pi_V(y))\geq \kappa$.
\end{enumerate}

\subsection{Totally geodesic submanifolds.} We now verify the axioms for an algebraic totally geodesic submanifold $N$, following the order introduced above. Note that we write $d_U(x,y)$ instead of $d_U(\pi_U(x), \pi_U(y))$, leaving implicit that to compute such distance we must first project to $\cC U$. 

\bold{Basic set up.} We start by determining what our set of domains will be. 

\begin{lemma}\label{L:dichotomy}
For any algebraic totally geodesic submanifold $N$ there exists a constant $D > 0$ such that for any subsurface $U$ the projection of $N$ to $\cC(U)$ is either unbounded or has diameter at most $D$. The projection is unbounded if and only if $U$ appears as a component of an unbounded simple factor. 
\end{lemma}

The requirement that the simple factor be unbounded is equivalent to the requirement that it not be a point. We say $U$ is a component of the simple factor if $\cT(U)$ is a factor of the product of Teichm\"uller spaces in which the simple factor lies. 

\begin{proof}
Suppose first that $U$ is annular. If the core curve does not get short, then the diameter bound follows from Rafi's theory of active intervals (recalled in Theorem \ref{T:Active}). Otherwise, the core curve gets short and, keeping in mind Lemma \ref{L:InvariantUnderTwist} and Corollary \ref{C:Pinchable}, we can pinch it and twist it to show unbounded image.

Now take $U$ to be non-annular. We proceed in cases, sometimes making use of Theorem \ref{T:main1}:

\begin{enumerate}
    \item \textit{Some boundary curve of $U$ is not a cylinder curve.} By Corollary \ref{C:Pinchable}, this curve cannot get short, so the theory of active intervals gives the diameter bound.
    \item \textit{All boundary curves of $U$ are cylinder curves, but one of these curves is $N$-equivalent to a curve $\alpha$ in $\cC U$.} The theory of active intervals gives that the projection to $\cC U $ can only change when the boundary of $U$ is short, but in that case, by Lemma \ref{L:ratio}, $\alpha$ is short, so the projection is coarsely $\alpha$, providing a uniform bound on the diameter. 
    \item \textit{It is possible to pinch the boundary of $U$ without pinching a curve in the interior of $U$.} In this case $U$ appears in some simple factor $F$. If this factor is a point, then the proof of Corollary \ref{C:CurveSS} shows that the projection to $\cC U$ has bounded diameter. If not, it contains a Teichm\"uller disc. Since, by \cite[Lemma 4.4]{RScovers} and \cite{KMS}, all Teichm\"uller discs have infinite diameter in the curve graph of the surface, this gives unbounded image. \qedhere 
\end{enumerate}
\end{proof}

\begin{definition}
Given an algebraic totally geodesic submanifold $N \subseteq \mathcal{T}_{g,n}$: 
\begin{enumerate}
\item We call the subsurfaces $U$ for which $\pi_U(N)$ is unbounded active.  
\item We call  two non-annular active subsurfaces equivalent if they occur in the same simple factor of the intersection of $\ol{N}$ with a stratum of the Deligne--Mumford bordification. 
\item We call two annular active subsurfaces equivalent if their core curves are equivalent. 
\end{enumerate}
\end{definition}

Note that Lemma \ref{L:SimpleFactorsWellDefined}  gives that two equivalent active non-annular subsurfaces  will be in the same simple factor of every intersection of $\ol{N}$ with a stratum of the Deligne--Mumford bordification where the two subsurfaces appear. Note also that, by definition, a non-annular subsurface cannot be equivalent to an annular subsurface.  

\begin{example}
As remarked in the introduction, the basic example of totally geodesic submanifolds are covering constructions in the sense of \cite[Section 6]{MMW}, resulting in a copy of a smaller Teichm\"uller space sitting inside of a larger Teichm\"uller space. In this case the active subsurfaces are  connected components of lifts of non-pants  
connected subsurfaces, and the different connected components of such a lift are equivalent. 
\end{example}

Let $\mfS_N$ be the set of equivalence classes of active subsurfaces of $N$. For each equivalence class $E=\{U_1, \ldots, U_k\}$ in $\mfS_N$ we let $\cC(E)$ be the image of $N$ in 
$$\cC(U_1) \times \cdots \times \cC(U_k).$$
Note that Lemma \ref{L:dichotomy} gives that $\cC(E)$ is unbounded.

\begin{lemma}
For any $E\in \mfS_N$ the space $\cC(E)$ is Gromov hyperbolic.
\end{lemma}

\begin{proof}
By Corollary \ref{C:CurveSS} and Theorem \ref{T:AnnularCurveSS}, the projections to each factor are quasi-isometric embeddings, so it suffices to show that $\pi_{U_1}(N)$ is Gromov hyperbolic.

If ${U_1}$ is non-annular this follows from the fact that Teichm\"uller geodesics map to unparametrized quasi-geodesics \cite[Theorem B]{RafiHyp}, which gives quasi-convexity and hence hyperbolicity of the image. If ${U_1}$ is annular one can explicitly show that $\pi_U$ is coarsely surjective, which directly gives hyperbolicity of the image.
\end{proof}

For more on the image of Teichm\"uller discs in the curve graph see \cite{Shadows}. 

\bold{Projections.} By definition we have natural maps $\pi_E : N \to \cC(E)$ defined as the restriction of the product map $(\pi_{U_1}, \ldots, \pi_{U_k})$ to $N$. The map $\pi_E$ is coarsely Lipschitz since the maps $\pi_{U_i}$ are coarsely Lipschitz.

\bold{Nesting.}
We say $E=\{U_1, \ldots, U_k\}$ is nested in  $F=\{V_1, \ldots, V_\ell\}$ if every $U_i$ is nested in a $V_j$ and every $V_j$ has a $U_i$ nested in it.
In this case we define $\rho^F_{E}(\alpha)$ to be the union of the $\pi_{E}(X)$ for $X\in N$ with $\pi_F(X)$ close to $\alpha$. 

Notice that $\pi_E$ is surjective by definition. We define $\rho^E_F$ to be the boundary components of the $U_i$ in the corresponding $V_j$. Since the boundary components of the $U_i$ can be made simultaneously short, this gives a well defined point in $\cC(F)$. 

We should compare our definition of $\rho^F_{E}(\alpha)$ to the usual definition used in the HHS structure of Teichm\"uller space. Let $U$ and $V$ be subsurfaces with $U$ nested in $V$. To start, we define $\rho^U_V$ to be the boundary of $U$, which is a coarse point in $\cC(V)$. The map $\rho^V_U : \cC V \to \cC^U$ is only really required to be meaningful for points a definite distance from $\rho^U_V$, but in this context there is a meaningful definition for all curves $\alpha$ that cut $U$, and one defines $\rho^V_U(\alpha)$ to be the usual subsurface projection of $\alpha$ to $U$. Our definition of $\rho^F_{E}(\alpha)$  is compatible with the usual definitions in the following sense. 

\begin{lemma}\label{L:rhoIsAsUsual}
Suppose $U$ is nested in $V$. Consider $\alpha\in \cC(V)$ far from $\rho^U_V$. Then the usual subsurface projection $\rho^V_U(\alpha)$ of $\alpha$ to $U$ lies within bounded distance of set of $\pi_U(X)$ where $X$ ranges over the set of $X\in \cT_{g,n}$ where $\pi_V(X)$ is close to $\alpha$. 
\end{lemma}

\begin{proof}
This follows from axiom  \eqref{O:TC:Functoriality} for the usual HHS structure on $\cT_{g,n}$. 
\end{proof} 

Our choice of $\rho^F_{E}(\alpha)$ was made so that it is immediate from the definition that it lies in $\cC(E)$, but Lemma \ref{L:rhoIsAsUsual} shows that we could have defined it in other ways. (Moreover \cite[Proposition 1.11]{HHS2}  shows that we could have even avoided defining $\rho^V_U$ in this case had we wished to.)

\bold{Orthogonality.}  We say $E=\{U_1, \ldots, U_k\}$ is orthogonal to  $F=\{V_1, \ldots, V_\ell\}$ if every $U_i$ is orthogonal to every $V_j$.

\bold{Transversality and consistency.} By definition, $E=\{U_1, \ldots, U_k\}$ and  $F=\{V_1, \ldots, V_\ell\}$ are transverse if they are not orthogonal and neither is nested in the other. To proceed, we need to better understand transverse pairs. 

\begin{lemma}
\label{L:transverse}
Let $\{U_1, \ldots, U_k\}$ and $\{V_1, \ldots, V_\ell\}$ be distinct equivalence classes of active subsurfaces of $N$. Then, one of the following holds: 
\begin{enumerate}
\item Every $U_i$ is nested in some $V_j$ and every $V_j$ has some $U_i$ nested in it; or the same holds with the roles of the $U_i$'s and $V_j$'s interchanged. 
\item Every $U_i$ is orthogonal to every $V_j$. 
\item Every $U_i$ is transverse to some $V_j$ and every $V_j$ is transverse to some $U_i$. 
\end{enumerate}
\end{lemma}

Here we use the usual terminology for the HHS structure on Teichm\"uller space, so $U_i$ is nested in $V_j$ means that either $U_i=V_j$, or $U_i$ is non-annular and is contained up to isotopy in $V_j$, or $U_i$ is annular with core curve a non-peripheral curve of $V_j$. Also, $U_i$ is orthogonal to $V_j$ if they are not equal and can be isotoped to be disjoint. As usual, $U_i$ is transverse to $V_j$ if they are neither nested nor orthogonal. More concretely, $U_i$ is transverse to $V_j$ if either their boundaries intersect or each contains a boundary curve of the other non-peripherally. 

\begin{proof}
We proceed in cases. 

\bold{Case 1: Both equivalence classes are non-annular.}
Suppose that the boundary of some $U_i$ intersects the boundary of some $V_j$. Then, Lemma \ref{L:BoundaryCurvesSimpleFactor} part \eqref{L:BoundaryCurvesSimpleFactor:3} 
and Corollary \ref{C:Simultaneous} give that the boundary of every $U_i$ must intersect the boundary of some $V_j$ and vice versa, so the third conclusion holds. 

We can thus assume that all the boundaries do not intersect. Lemma \ref{L:BoundaryCurvesSimpleFactor} part \eqref{L:BoundaryCurvesSimpleFactor:2} gives that if some $U_i$ contains a boundary curve $\gamma$ of a $V_j$ non-peripherally, then all $U_i$ contain a boundary curve of a $V_j$ non-peripherally.
By symmetry, we also get the claim with the role of the $U_i$'s and $V_j$'s swapped. 

Now suppose that some $U_i$ contains non-peripherally a boundary component of one of the $V_j$ and that some $V_j$ contains non-peripherally a boundary component of the $U_i$. In this case the previous paragraph gives that the third conclusion holds. 

It thus suffices to consider the case when some $U_i$ contains non-peripherally a boundary component of a $V_j$, but no $V_j$ contains non-peripherally a boundary component of a $V_i$. In this case the first conclusion holds, since if $U$ and $V$ are connected subsurfaces with disjoint boundary and $V$ contains  a boundary component of $U$ but $U$ does not contain a boundary component of $V$, then $V$ contains $U$. 

\bold{Case 2: $\{U_1, \ldots, U_k\}$ is annular and $\{V_1, \ldots, V_\ell\}$ is non-annular.} In this case each $U_i$ can be thought of as a curve as well as an annular domain. If one of the $U_i$ is contained non-peripherally in one of the $V_j$, then Lemma \ref{L:BoundaryCurvesSimpleFactor} part \eqref{L:BoundaryCurvesSimpleFactor:2} gives that that the first conclusion holds. 
If one of the $U_i$ intersects the boundary of one of the $V_j$, then Corollary \ref{C:Simultaneous} and Lemma \ref{L:BoundaryCurvesSimpleFactor} part \eqref{L:BoundaryCurvesSimpleFactor:3} give that the third conclusion holds. 

\bold{Case 3: Both equivalence classes are annular.} In this case Corollary \ref{C:Simultaneous} gives that either the second or third conclusion hold.
\end{proof}

It follows from Lemma \ref{L:transverse} that if $E=\{U_1, \ldots, U_k\}$ and  $F=\{V_1, \ldots, V_\ell\}$ in $\mfS_N$ are transverse, then each $U_i$ is intersected by the boundary of some $V_j$ and vice versa. In particular, we can define $\rho^E_F$ using these intersections as in the usual HHS structure on Teichm\"uller space. Since the boundary components of the $U_i$ can be made simultaneously short, this gives a point of $\cC(F)$. 

With all relevant definitions complete, we observe that axiom \eqref{O:TC:Behrstock} follows immediately from the corresponding axiom for Teichm\"uller space. Here one must keep in mind the diagonal like behavior of $\cC(E)$ inside of $\prod \cC(U_i)$, which gives that two points in $\cC(E)$ are close if and only if they are close in any given factor $\cC(U_i)$. So in particular, if $d_{U_i}(x, \rho^{V_j}_{U_i})$ is small for a single $i,j$, then $d_E(x, \rho^F_E)$ is small.  

A similar shows \eqref{O:TC:Functoriality}, given Lemma \ref{L:rhoIsAsUsual}. Axiom \eqref{O:TC:Rhos} also follows from the corresponding axiom for Teichm\"uller space because we have defined the relevant $\rho$'s for $N$ using the corresponding $\rho$'s for Teichm\"uller space.

\bold{Finite complexity.} This follows immediately from the previous definitions.

\bold{Large links.} 
%
%
%
We assume $E$ is larger than the constant in Lemma \ref{L:dichotomy}.
Fix an equivalence class $\{W_1, \ldots, W_r\} \in \mfS_N$ and $x, x'\in \cT_{g,n}$. The large links axiom for Teichmuller space gives that we can find a set of at most $\lambda d_{W_1}(x, x') + \lambda$ domains properly nested in $W_1$ that contain all $T$ properly nested in $W_1$ with $d_T(x, x')>E$.  If these domains are active, then considering their equivalence classes proves large links (in this instance) for $N$. Otherwise, if one such domain $U$ is not active, then Lemma \ref{L:dichotomy} gives that $d_U(x,x')$ is uniformly small, so applying the large links axiom again for Teichmuller space gives a uniformly finite collection of domains properly nested in $U$. Iteratively applying large links in this way, we eventually obtain a (possibly empty) uniformly finite set of domains properly nested in $U$ that are active and which serve to prove the large links axiom for $N$.  Compare to the Passing Up Lemma \cite[Lemma 2.5]{HHS2}.  

\bold{Bounded geodesic image.} This follows directly from the corresponding axiom for Teichm\"uller space, keeping Lemma \ref{L:rhoIsAsUsual} in mind. 

\bold{Partial realization.} Fix a pairwise orthogonal family of domains in $\mfS_N$ and a point in each of the corresponding products of curve graphs. To start, assume none of the domains are annular. By the proof of Corollary \ref{C:CurveSS}, each point can be realized as a point in the corresponding Teichmüller space after pinching all the boundary curves of every surface in every domain of the family. Using semisimplicity, it follows that the tuple of points can be realized simultaneously on the corresponding pinched surface. Unpinching the nodes we realize the tuple by a point of $N$ where all the corresponding boundary curves are short; see Corollary \ref{C:UnPinchOneAtATime}. This point proves the partial realization axiom when no annular domains are considered. 

The proof when some domains are annular is similar and we only provide a sketch. First pinch and arrange for all non-annular domains to realize the appropriate projections. Then, keeping in mind Lemma \ref{L:DilationLemma} as well as the proof of Lemma \ref{L:OpenUpNoProjChange}, one can show that it is possible to unpinch so that each equivalence class of previously pinched curves is approximately any desired small length, and without changing the subsurface projections to the relevant non-annular domains. Then one can apply Dehn multi-twists to adjust the $x$-coordinates of the image in the curve complexes of the annular domains as required. 

\bold{Uniqueness.} This follows immediately from the uniqueness axiom for Teichm\"uller space and the uniform bound provided by Lemma \ref{L:dichotomy}.

\bold{Conclusion.} This concludes the proof that algebraic totally geodesic submanifolds of Teichmüller space are hierarchically hyperbolic. 

\subsection{Fundamental groups.} 
Let $\Gamma_N$ denote the stabilizer of $N$ in the mapping class group. Our goal is now to show that $\Gamma_N$ is a hierarchically hyperbolic group. We will only sketch this, leaving some details to the reader. Recall that a group is a hierarchically hyperbolic group if it has a suitably nice action on a hierarchically hyperbolic space. This action in particular must be metrically proper and cobounded, and induce an action on the set of domains that is cofinite. See \cite[Definition 1.21]{HHS2} for the full definition.

Note that $\Gamma_N$ is finitely  presented because it is the fundamental group of an algebraic variety. To establish that $\Gamma_N$ is a hierarchically hyperbolic group, we will fix a finite generating set and consider the action of $\Gamma_N$ on its Cayley graph, which we denote by $\Cay(\Gamma_N)$. 

The action of $\Gamma_N$ on $\Cay(\Gamma_N)$ is certainly metrically proper and cobounded. We will describe a hierarchically hyperbolic structure on $\Cay(\Gamma_N)$ and then complete the proof that $\Gamma_N$ is a hierarchically hyperbolic group.

\bold{Basic set up.} The set of domains will be the same as for $N$, which is to say it will still be $\mfS_N$. The set of associated hyperbolic spaces will be almost the same: For $E$ non-annular, we use the same hyperbolic space $\cC(E)$. When $E$ is annular, the hyperbolic space we used above is quasi-isometric to a horoball, and we will now replace this horoball with its boundary together with the path metric on its boundary. That is to say, when $E$ is annular,  the new hyperbolic space can be taken to be the real line. We will abuse notation and continue to call this $\cC(E)$.

\bold{Projections.} The projection maps are obtained as follows. Fix $X_0\in N$. It may be helpful to think of $X_0$ as being in the thick part of $N$, but of course every point of $N$ is in the $\e$-thick part for some $\e > 0$. 

We then define $\pi_E(g)$ to be $\pi_E(gX_0)$. Since $gX_0$ is uniformly thick, if $E$ is annular, this projection will lie uniformly close to the boundary of the horoball, so this gives a coarsely well defined point even with our new version of annular curve graphs. It will be implicit in our discussion of partial realization that these projections are coarsely surjective. 
%
%

\bold{Nesting, orthogonality, and finite complexity.} We are using the same set of domains as for $N$ and we leave the definitions of nesting and orthogonality unchanged. 

\bold{Transversality and consistency, large links, and bounded geodesic image.} This follows immediately from the corresponding statements for $N$. 

\bold{Partial realization.} Given a collection of orthogonal domains and a point in each of their associated hyperbolic spaces, we use partial realization for $N$ to get an associated point $X\in N$. We then modify this point, as in the proof of Lemma \ref{L:OpenUpNoProjChange}, to be uniformly thick. We continue to denote the modification by $X$. 
Since the $\Gamma_N$ action on the thick part of $N$ is co-compact, there exists $g\in \Gamma$ such that $g X_0$ is uniformly close to $X$, and this $g$ shows that the partial realization holds for the tuple under consideration. 

%
%

\bold{Uniqueness.} Consider $g, h \in \Gamma_N$ and suppose $g$ and $h$ are far apart in $\Gamma_N$. Because the action of the mapping class group on Teichm\"uller space is metrically proper, the action of $\Gamma_N$ on $N$ is metrically proper, so it follows that $gX_0$ and $hX_0$ are far apart. We then get uniqueness for $\Gamma_N$ from uniqueness for $N$, keeping in mind that the inclusion of the boundary of a horoball into the horoball is metrically proper. 

\bold{Conclusion.} This conclude our sketch that $\Cay(\Gamma_N)$ is hierarchically hyperbolic. The fact that $\Gamma_N$ acts on $N$ easily implies it acts on $\mfS_N$. Since each $\Gamma_N$ orbit for this action gives a component of the boundary of $\ol{\pi(N)}$ in $\ol{\cM}_{g,n}$, the fact that $\pi(N)$ is a variety gives that there are only finitely many orbits by Corollary \ref{C:algebraic1}.

\subsection{The container axiom.} We conclude by commenting a little bit more on the difference between an HHS and an almost HHS. The definition of an almost HHS includes the following as axiom \eqref{O:BoundedOrth}.

\bold{The Bounded Orthogonality Axiom:} The cardinality of a set of pairwise orthogonal domains is uniformly bounded. \medskip  

\noindent In contrast, the orthogonality axiom of an HHS includes the following.  

\bold{The Container Axiom:} For each domain $T$, and each domain $U$ nested in $T$ for which there exists a domain $V$ nested in $T$ and orthogonal to $U$, there exists a domain $W$ properly nested in $T$ which contains all such $V$. 

\medskip

If the so called container domain $W$ can be additionally chosen to be orthogonal to $U$ then the HHS is said to have clean containers \cite[Definition 7.1]{AlmostHHS} and the $W$ can be thought of as the perp of $U$ inside $T$. 

If one uses a slightly larger set of domains than we use above, it is easy to directly verify that the resulting HHS or HHG has clean containers. To do this one only has to make new domains for every orthogonal tuple of old domains; this  roughly allows disconnected subsurfaces as well as connected ones. The hyperbolic spaces of the new domains should be viewed as (coarse) points. Given a new domain, which is an orthogonal tuple of the old domains, it has a clean container that can be obtained as follows. Consider the complement of all the boundaries of all the subsurfaces involved. First take the union of the active domains in the complement that do not contain any of the subsurfaces involved. Then add on the annular boundary domains that were not part of the orthogonal tuple. 

This has some relevance that we have suppressed above, since \cite[Remark A.7]{AlmostHHS} has an omission and in fact there is a mild complication in the equivalence of hierarchically hyperbolicity and almost hierarchically hyperbolicity for HHGs \cite[Remark 3.4]{ForRemark}. 

\section{The proof of Theorem \ref{T:SummaryForClassification} and open questions}\label{S:concluding}

\subsection{About the proof of Theorem \ref{T:SummaryForClassification}} We can now point out where the various parts of this theorem were proved in the paper.

\begin{proof}[Proof of Theorem \ref{T:SummaryForClassification}]
Part \eqref{SFC:CylinderEverywhere} follows from Corollary \ref{C:StrebelInQN}; part \eqref{SFC:EquivAlsoCylinder} from Corollary \ref{C:Connected}; part \eqref{SFC:SimultaneousRealization} from Lemma \ref{L:Simultaneous}.

We check  part \eqref{SFC:Pinchable} as follows. Given a pinchable multicurve,  Corollary \ref{C:Pinchable} and Lemma \ref{L:ratio} show that it is a union of equivalence classes of curves in $\cS$. Conversely, given a union of disjoint equivalent classes of curves in $\cS$, we can argue as in Corollary \ref{C:StrebelInQN} (using a product of twists given by Lemma \ref{L:InvariantUnderTwist}) to find a differential where all the curves are simultaneously cylinders and then pinch these curves as in Corollary \ref{C:Pinchable}. 

The final claim, part \eqref{SFC:BoundaryOfFilled}, is proven as follows. There is nothing to prove if the two equivalences classes of curves are disjoint, so assume they are not disjoint. Let $U_1, \ldots, U_k$ be the connected components of the subsurface filled by the two equivalence classes. Corollary \ref{C:Simultaneous} gives that none of the $U_i$ are annular. 

Lemma \ref{L:InvariantUnderTwist} gives multi-twists associated to each of the two equivalence classes. A variant of the Thurston--Veech construction gives a mapping class $g$ which is supported on the union of the $U_i$'s, which stabilizes $N$, and which is a pseudo--Anosov on each of the $U_i$'s. 

Fix $X\in N$ and let $Y=g^p X$ for $p$ very large. In particular, $d_V(X, Y)$ is large if $V$ is one of the $U_i$'s and otherwise it is bounded. It follows from the definition that $\{U_i\}$ is a union of equivalence classes of active subsurfaces. Lemma \ref{L:dichotomy} then gives that the boundary of the $U_i$'s can be pinched.  The fact that the boundary of the $U_i$ is a union of equivalence classes of curves is Lemma \ref{L:BoundaryCurvesSimpleFactor} part \eqref{L:BoundaryCurvesSimpleFactor:3}. 
\end{proof}

\subsection{Open questions} Of the many open questions related to this paper, we highlight first the following: 

\begin{question}
Is it true that every totally geodesic submanifold of dimension at least 3 is a covering construction in the sense of \cite[Section 6]{MMW}, based off an entire Teichm\"uller space? (Compare to \cite[Question 1.3]{HighRank}).
\end{question}

We also highlight just a small number of instances of Metaconjecture \ref{M} to emphasize the variety of perspectives that might be interesting here. 

\begin{question}
Is there a version of the curve complex for higher dimensional totally geodesic submanifolds that is homotopy equivalent to a wedge of spheres? Are the fundamental groups of totally geodesic subvarieties virtual duality groups? (Compare to \cite{Harer}).
\end{question}

\begin{question}
Are all algebraic totally geodesic submanifolds biholomorphic to bounded domains?
\end{question}

\begin{question}
Can one give nice, explicit presentations for the fundamental groups of the examples in \cite{MMW, EMMW}?
\end{question}

We also mention some open questions that are motivated at least in part by comparison to Teichm\"uller curves. 

\begin{question}
Are the surface bundles and surface group extensions associated to totally geodesic subvarieties hierarchically hyperbolic? (Compare to \cite{ExtensionsII}).
\end{question}

\begin{question}
Are algebraic totally geodesic subvarieties rigid? (Compare \cite{McMullenRigid} and note the potential relevance of results along the lines of those in \cite{InfiniteEnergy}). 
\end{question}

\appendix

\section{Annular projections using cylinders (after Rafi)}\label{A:Rafi}

In this appendix we continue to fix a Teichm\"uller space $\cT_{g,n}$, and all the constants we produce continue to implicitly depend on $g$ and $n$. See Section \ref{SS:AnnularSetup} for background on annular curve graphs. 

As we recall in the next remark, the usual subsurface projection to annular subsurfaces does not behave as nicely as one would like along Teichm\"uller geodesics. The purpose of this appendix is to communicate a beautiful idea of Rafi which for some purposes sidesteps this difficulty. We thank Rafi for sharing this with us. The results of this section are only used in Appendix \ref{S:electrification}. 

\begin{remark} Recall from Section \ref{SS:AnnularSetup} that there are two versions of annular curve graphs. If one uses the projections to the annular curve graphs that are quasi-isometric to $\bZ$, Rafi conjectures there is no backtracking, but this conjecture remains open as of the writing of this paper. If, as we do, one uses the version of annular curve graphs that are horoballs, it is well known that there can be backtracking.\footnote{More specifically, the image of a Teichm\"uller geodesic can make an excursion deep into the horoball $\{z: \Im(z)\geq 1\}$  starting at a point on the boundary $\{z: \Im(z)= 1\}$ and eventually returning to roughly the same point on the boundary. To see this, take for example a flat torus that generates a cobounded geodesic in $\cT_{1}$. Take two copies of this torus, and glue them along an arbitrarily tiny slit of slope 1. The slit curve is hyperbolically short, but if one flows in either direction for a sufficiently long time it is not hyperbolically short. One can use \cite[Theorem 4.3]{RafiComb} to verify that all points where the slit curve is not hyperbolically short have roughly the same twisting.} 

Since there is potential for confusion, we pause briefly to comment on Rafi's results related to this. Rafi has proven proven a no backtracking result,  which, as stated in \cite[Theorem B]{RafiHyp}, does not explicitly rule out annular subsurfaces. And indeed Rafi does  prove a no backtracking result for some definition of twisting. However the definition of twisting Rafi uses is not what is used in defining maps from Teichm\"uller space to annular curve graphs, and instead is part of what we use in this appendix. See \cite[Theorem 4.3]{RafiComb} for a comparison and \cite[Section 2.7]{StatHyp} for some related discussion.
\end{remark}

The idea to sidestep this difficulty is to define a variant of subsurface projection for annular subsurfaces. The advantage of this variant is that one simultaneously maintains a version of the uniqueness axiom (axiom \eqref{O:Uniqueness}) and gets  non-backtracking.  The disadvantage of this variant is that it does not give a well defined function on all of Teichm\"uller space. Instead, for each Teichm\"uller geodesic, one gets a projection defined on this Teichm\"uller geodesic, which makes use in an essential way of the quadratic differentials generating Teichm\"uller geodesic. Equivalently, one gets a projection defined on the bundle of non-zero quadratic differentials.  

\begin{definition}
For any quadratic differential $q$ and any simple closed curve $\alpha$, define $\pi_\alpha(q)$ to be the point in the horoball $\Im(z)\geq 1$ in the upper half plane $\bH$ as follows.  If $q$ does not have a cylinder with core curve $\alpha$, set $\pi_\alpha(q) = 0+i$. If $q$ does have a cylinder with core curve $\alpha$, 
set 
$$\pi_\alpha(q) = \twist_q(\alpha) + \max(1,\modulus_q(\alpha)) i ,$$ 
where typically $\twist_q(\alpha)$ is the twisting at $\alpha$ between the vertical foliation of $q$ and a flat perpindicular to the core curve of $\alpha$, as defined in \cite[Section 4]{RafiComb}, and $\modulus_q(\alpha)$ is the modulus of the cylinder with core curve $\alpha$. In the exceptional situation when the cylinder is vertical, we define $\twist_q(\alpha)=0$.  
\end{definition}

We warn that special care should be taken with horizontal and vertical cylinders. It is safest to use this definition along one Teichm\"uller geodesic at a time, in which case it is easy to understand.

If $q_t$ is a Teichm\"uller geodesic and $\alpha$ is not vertical or horizontal, then  \cite[Equation (16)]{RafiComb} and  \cite[Equation(4)]{RafiHyp} give that there are constants $T_\alpha$ (the total twisting along the geodesic) and $t_\alpha$ (the balanced time) such that 
$$\twist_q(\alpha) \approxadd \frac{T_\alpha e^{2(t-t_\alpha)}}{4\cosh^2(t-t_\alpha)-2}$$ and $$\quad 
\modulus_q(\alpha) \approxmult \frac{T_\alpha}{\cosh^2(t-t_\alpha)} \approxmult \frac{T_\alpha}{4\cosh^2(t-t_\alpha)-2}.$$
Here, as in \cite{RafiHyp}, $A\approxadd B$ means $|A-B|$ is bounded above by a constant only depending on $g$ and $n$, and  $\smash{A\approxmult B}$ means $A/B$ is bounded above and below by constants only depending on $g$ and $n$.

\begin{lemma}\label{L:quasigeodesic}
The path 
$$\frac{e^{2t}+i}{e^{2t}+e^{-2t}}$$
is geodesic of speed 2 in the upper half plane $\bH$. 
\end{lemma}

This path is illustrated in Figure \ref{F:RafiFormulas}.

\begin{figure}[h!]
\includegraphics[width=0.45\linewidth]{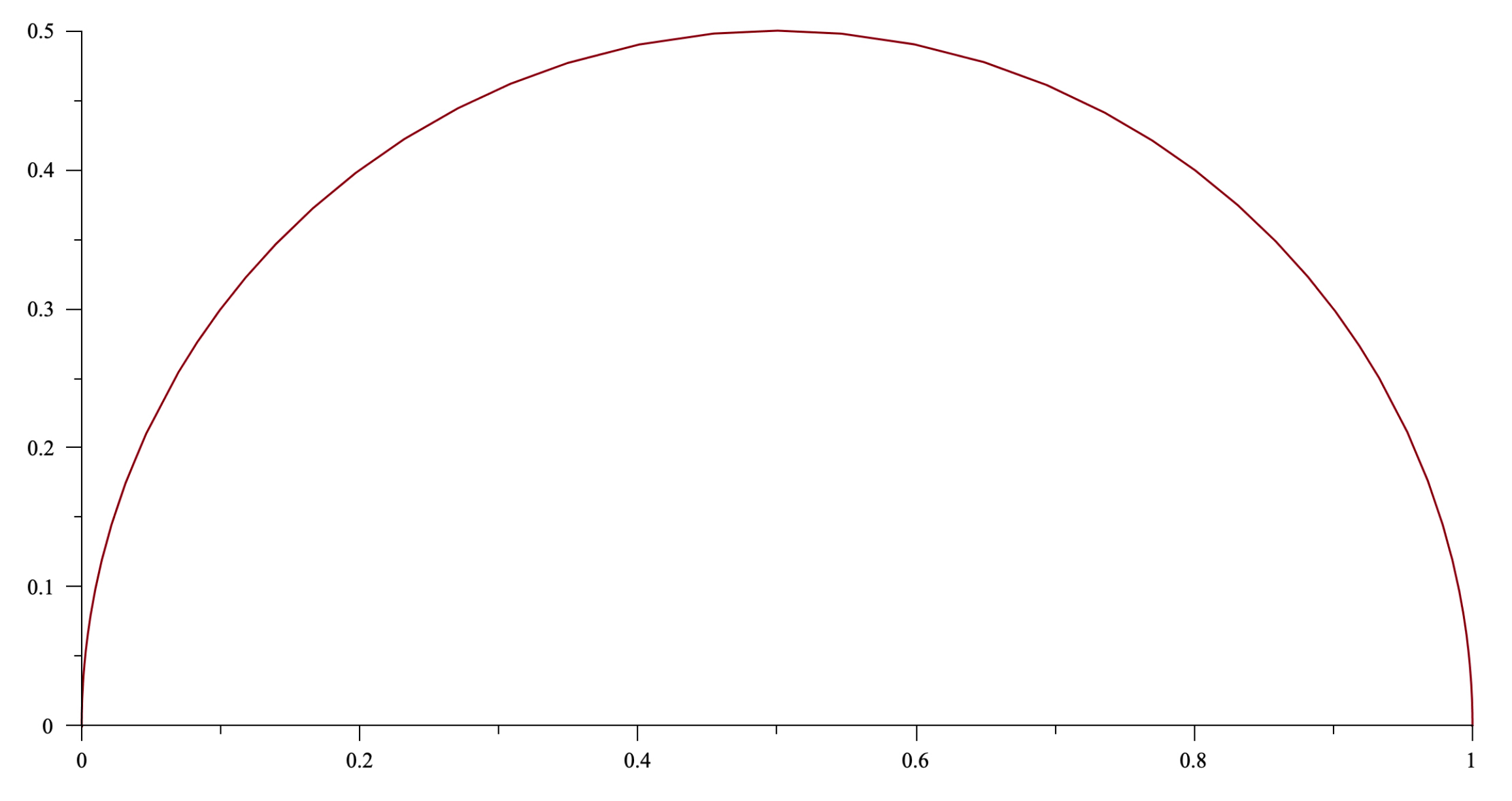}
\caption{The path of Lemma \ref{L:quasigeodesic}.}
\label{F:RafiFormulas}
\end{figure}  

\begin{proof}
Consider the vertical geodesic $t \mapsto i e^{2t}$ and apply the M\"obius transformation $z \mapsto z / (z + 1).$ (We thank Saul Schleimer for pointing out this proof to us.)   
%
%
\end{proof}

\begin{corollary}\label{C:NewNoBacktracking}
If $q_t$ is a Teichm\"uller geodesic, then $\pi_\alpha(q_t)$ is an unparametrized quasi-geodesic. 
\end{corollary}

\begin{proof}
In the special case when $\alpha$ is vertical or horizontal, this follows easily because the twist is constant and an even simpler formula governs the modulus. The typical case follows from Lemma \ref{L:quasigeodesic} by using the hyperbolic isometry $z\mapsto T_\alpha z$. 

A technical point is that the max in the formula for $\pi_\alpha$ does not cause problems; a priori one might worry that significant change to the real part might take place during an interval when the imaginary part is constantly equal to 1. For this one should note that, with the formula in Lemma \ref{L:quasigeodesic}, after $t$ gets negative enough to have that the imaginary part $1/(e^{2t}+e^{-2t})$ has size at most $\delta=1/T_\alpha$, the $x$ coordinate is within $e^{2t} \delta < \delta$ 
of its limiting value. A similar observation holds when $t$ is large and positive.
\end{proof}

In the next lemma we give a uniqueness constant $\theta_u$ following the usual notation in the uniqueness axiom; the subscript is just a reminder that this relates to uniqueness. 

\begin{proposition}\label{P:KasraUnique}
Fix a Teichm\"uller space $\cT_{g,n}$. 
 For all $\kappa$ there exists $\theta_u$ such that if $X, Y\in \cT_{g,n}$ have  $d(X,Y)\geq \theta_u$ then either 
 \begin{enumerate}
 \item there exists a non-annular $V$ with $d_V(X,Y)\geq \kappa$, or 
 \item if $q_X$ and $q_Y$ are the differentials such that $g_T(X, q_X)=(Y, q_Y)$ for some $T\in \bR$, then there is a curve $\alpha$ with 
 $$d_{\bH}(\pi_\alpha(q_X), \pi_\alpha(q_Y))\geq \kappa.$$
 \end{enumerate}
\end{proposition}

Before giving a proof we recall a standard lemma, which can be derived from from Rafi's thick-thin decomposition (with specific ingredients including \cite[Theorem 4]{ThickThin} and \cite[Theorem 3.1]{RafiHyp})
%
%
or from more recent technology such as \cite{Compactification1} (using that $q$ is close to a limit differential with no horizontal nodes and hence more than one level).

\begin{lemma}\label{L:ThickThin}
For any $D$ and any $M$, there exists an $\e$ such that if $q \in Q\cT_{g,n}$ has no cylinder of modulus greater than $D$ and has a curve of hyperbolic length less than $\e$, then there is a subsurface $U$ of $q$ bounded by expanding annuli of modulus at least $M$. 
\end{lemma} 

Here the expanding annuli are taken to be contained in $U$. This $U$ is a small subsurface, and the expanding annuli mediate between the small scale of $U$ and larger scale adjacent to $U$.

\begin{proof}[Proof of Proposition \ref{P:KasraUnique}]
The cases of $\cT_{1,1}$ and $\cT_{0,4}$ are straightforward to analyze, either directly or using the ideas below, so we can induct on $\dim \cT_{g,n}$ and assume that the complex dimension of $\cT_{g,n}$ is at least two and the result is known for smaller dimensional Teichm\"uller spaces. 

Consider the path $g_t(X, q_X), t\in [0,T]$. If at some point this has a cylinder of large modulus with core curve $\alpha$, then we can directly check that (2) holds using Corollary \ref{C:NewNoBacktracking} and the related formulas, which together give that the modified annular projection must change when the modulus is large, and that this change cannot be undone. So suppose there is an upper bound on the modulus of cylinders along this path. 

If the path stays in the thick part, it makes uniform progress in the main curve graph \cite[Lemma 4.4]{RScovers}. So it suffices to consider the case when a curve gets very short. In this case Lemma \ref{L:ThickThin} and the results of \cite[Section 4]{RafiHyp} on projection of a quadratic differential to subsurfaces give that there is a long subinterval where the path follows a geodesic in a smaller dimensional Teichm\"uller space. The result then follows from the induction hypothesis, keeping in mind Rafi's no backtracking result \cite[Theorem B]{RafiHyp}.
\end{proof}

\section{Electrified Teichm\"uller geodesics}\label{S:electrification}

In this appendix we continue to fix a Teichm\"uller space $\cT_{g,n}$, and all the constants we produce continue to implicitly depend on $g$ and $n$. Given a Teichm\"uller geodesic $$X_t, t\in \bR$$ and a constant $D>0$, for each (essential, non-peripheral) simple closed curve $\alpha$ we define $I_\alpha \subset \bR$ to be the convex hull of the $t\in \bR$ such that $\alpha$ has length at most $D$ at $X_t$. We define $\cE$ to be the electrification of $\bR$ along these intervals. The purpose of this section is to prove the following. Since this result does not seem to have been previously recorded in  the literature we provide a proof here, making use in particular of Rafi's ideas from the previous appendix. 

\begin{theorem}\label{T:electrification}
For any $D>0$ there exists $K$ such that for any Teichm\"uller geodesic $X_t, t\in \bR$ the natural map 
$$\pi : \cE \to \cC(\Sigma)$$
is a $K$-quasi-isometric embedding. 
\end{theorem}

Here $\Sigma$ is the surface of genus $g$ with $n$ punctures. Note that the fact that $\pi$ is coarsely Lipschitz follows from the standard fact that the subset of $\cT_{g,n}$ where a curve $\alpha$ has length at most $D$ maps to a subset of $\cC(\Sigma)$ that is contained in a  ball centered at $\alpha$ of radius a function of $D$. That standard fact is implicit in the result of Masur and Minsky which proves that the curve graph is equal to a similar electrification of Teichm\"uller space \cite[Lemma 7.1]{MMI}.

When we apply Theorem \ref{T:electrification} it is helpful to also have the following.  

\begin{corollary}\label{C:electrification}
For any $D_2\geq D_1>0$ there exists $K$ such that the following holds. Consider a Teichm\"uller geodesic $X_t, t\in \bR$. For each $\alpha$, let $I_{\alpha} \subset \bR$ be an interval such that 
$$t\in I_\alpha \implies \ell_\alpha(X_t)\leq D_2, \quad\quad\text{and}\quad\quad t\notin I_\alpha \implies \ell_\alpha(X_t)\geq D_1.$$
Let $\cE$ be the electrification of $\bR$ along these intervals. Then the natural map 
$$\pi : \cE \to \cC(\Sigma)$$
is a $K$-quasi-isometric embedding. 
\end{corollary}

\begin{proof}[Proof of Corollary \ref{C:electrification} assuming Theorem \ref{T:electrification}]
Let $\cE_1$ be the electrification of $\bR$ used in Theorem \ref{T:electrification} with $D=D_1$, and let $\pi_1: \cE_1 \to \cC(\Sigma) $ be the associated map to the curve graph. Theorem  \ref{T:electrification} and the discussion above gives that there is a constant $K$ depending only on $D_1$ and $D_2$ such that $\pi_1$ is a $K$-quasi-isometric embedding and $\pi$ is $K$-Lipschitz. 

\begin{figure}[h!]
\begin{tikzpicture}

  \node (e1) at (0, 0) {$\mathcal{E}_{1}$};
  \node (e) at (2, 0) {$\mathcal{E}$};
  \node (xi2) at (2, -2) {$\cC(\Sigma)$};

  \draw[->, -Latex] (e1) -- node[above] {$\iota_{1}$} (e);
  \draw[->, -Latex] (e1) -- node[midway,below left] {$\pi_{1}$} (xi2);
  \draw[->, -Latex] (e) -- node[midway, right] {$\pi$} (xi2);

\end{tikzpicture}
\caption{The proof of Corollary \ref{C:electrification}.}
\label{F:factor}
\end{figure}  

We also have a $1$-Lipschitz inclusion $\iota_1 : \cE_1 \to \cE$, and $\pi_1 =   \pi \circ \iota_1$. The situation is summarized in Figure \ref{F:factor}. The result follows since $\iota_1$ is coarsely surjective. 
\end{proof}

\bold{Previous work.} 
Theorem \ref{T:electrification} is  related to work of Rafi and Schleimer \cite{RScovers}. It seems natural to expect that one could approach Theorem \ref{T:electrification} using the analysis in  \cite{RScovers}, but we do not follow that  approach here. 

We note that Theorem \ref{T:electrification} can be used to reprove the main result of \cite{RScovers}. Another proof of the main result of \cite{RScovers} was given in \cite{Tang}. See also \cite{Tang2} for a related result,  \cite{FuterSchleimer} for a version for arc graphs instead of curve graphs, and \cite{AougabPatelTaylor} for effective results. Some of these results use an ``electric distance" on hyperbolic three manifolds and hence have a similarity in flavor to Theorem \ref{T:electrification}; see \cite[Theorem 2.1.4]{BowditchELT} and  the discussion after the statement of \cite[Theorem 4.1]{AougabPatelTaylor} for more context on this.

See also \cite[Theorem 1.4]{AffineUndistorted} for a result of a similar flavor to Theorem \ref{T:electrification} which gives information on Teichm\"uller discs with sufficiently non-trivial Veech groups. 

\bold{Active intervals.} We start by recalling some of Rafi's theory of active intervals, established in \cite[Section 3]{RafiComb}. See also  \cite[Theorem 3.22]{DowdallMasur} and \cite[Section 2.5]{StatHyp}, for example, for summaries of this theory. It gives in particular the following. 

\begin{theorem}[Rafi]\label{T:Active}
There exists $\epsilon_0$ such that for  $0<\epsilon<\epsilon_0$  there exists $\epsilon'\in (0, \epsilon)$ and $M_{\rm{active}} >0$ such that for any proper subsurface $U$ and any Teichm\"uller geodesic $X_t$ there is an interval $I_U$ of times such that 
\begin{enumerate}
\item all components of $\partial U$ have length at most $\epsilon$ at every $X_t$ with  $t\in I_U$, 
\item at each $X_t, t\notin I_U$ at least one component of $\partial U$ has length at least $\epsilon'$, and 
\item for any $s,t$ in the same component of $\bR-I_U$, we have $d_U(X_s, X_t)\leq M_{\rm{active}} $. 
\end{enumerate}
\end{theorem}

This theorem allows the $U$ to be annular, and in this case we use the version of $\cC(U)$ that is a horoball in the hyperbolic plane. Note that $I_U$ can be empty or (in unusual cases) half-infinite. 

Rafi \cite[Theorem B]{RafiHyp} also proved  that the subsurface projection of a Teichm\"uller geodesic to a non-annular curve graph gives an unparametrized quasi-geodesic (a quasi-geodesic composed with a monotone time reparametrization).

\bold{Metrically proper maps.} A map $p:S \to T$ is called metrically proper if for any $d_0$ there exists a $c_0$ such that 
$$d(p(s_1), p(s_2))< d_0  \implies  d(s_1, s_2)< c_0.$$
Here it suffices to consider $d_0$ larger than any fixed constant. 

Dowdall and Taylor observed that a metrically proper alignment preserving map must be a quasi-isometric embedding \cite[Lemma 3.4]{DT}.  The definition of alignment preserving is such that a coarsely Lipschitz map to a Gromov hyperbolic space  that sends geodesics to unparametrized quasi-geodesics is alignment preserving. 

\bold{Reduction to a key proposition.} The following will be the key to the proof. 

\begin{proposition}\label{P:key}
For any $L$ and any  $D$  there exists $m$ and $B$ such that if $X,Y \in \cT_{g,n}$ have 
$$d_\Sigma(X,Y)<L$$
then there are $\alpha_1, \ldots, \alpha_m\in \cC(\Sigma)$ and subintervals $I_1, \ldots, I_m$ of the Teichm\"uller geodesic $[X,Y]$ from $X$ to $Y$ such that at each point in $I_j$ we have the length of $\alpha_j$ is at most $D$, and such that each subinterval of the complement of $\cup I_j$ has length at most $B$. 
\end{proposition}

Before we prove the proposition, we explain why it is sufficient. 

\begin{proof}[Proof of  Theorem \ref{T:electrification} assuming Proposition \ref{P:key}]
The map $\cE(\cI) \to \cC(\Sigma)$ is alignment preserving, so it suffices to prove it is metrically proper, and that follows immediately from the key proposition. 
\end{proof}

\bold{A uniformly finite collection of subsurfaces.}  We now give a construction  reminiscent of \cite{DowdallMasur} that captures some of the behavior of initial and terminal parts of a Teichm\"uller geodesic segment.

Since the next lemma is technical, we motivate it by briefly hinting at how it will be applied.  Our goal is now to prove Proposition \ref{P:key}.  The theory of active intervals gives that, given a subsurface $U$, during the interval of times in which a Teichm\"uller geodesic segment is far from its initial and terminal values in $\cC(U)$ the boundary of $\cC(U)$ will be short. During the time in which this boundary is short, we do not require any further information. 

\begin{lemma}\label{L:DM}
There exists a $C>0$ such that 
for any $M$ there exists an $N$ such that for any $X,Y\in \cT_{g,n}$ there exists a set $\cS=\cS(X,Y)$ of subsurfaces with the following properties. 
\begin{enumerate}
\item $\Sigma \in \cS$. 
\item $\cS$ has cardinality at most $N$.
\item Suppose $U\in \cS$ and $V$ is properly nested in $U$. Suppose there is large progress in $V$, in the sense that 
$$d_V(X,Y)>C.$$
Suppose also that this progress happens in the initial or terminal part of the Teichm\"uller geodesic as measured in $\cC(U)$, in the sense that
$$\min( d_U(\rho^V_U, X), d_U(\rho^V_U, Y) )  <M.$$
Then there exists $V'\in \cS$ properly nested in $U$ with $V$ nested in  $V'$. 
\end{enumerate}
\end{lemma}

Here $\rho^V_U$ continues to denote the boundary of $V$ viewed as a coarse point of $\cC(U)$. A consequence of the final condition is that for any surface $V$ not in $\cS$ where large progress occurs, there is a bigger subsurface $U$ in $\cS$ which proves that this progress does not happen in the initial or terminal part of the Teichm\"uller geodesic, as can be seen by looking at a minimal element in $\cS$ that contains $V$.

The proof will use the large links axiom, which gives constants $E, \lambda$ such that if $X$ and $Y$ are points of Teichm\"uller space and $U$ is a subsurface then there is a collection $Q=Q_U(X,Y)$ of at most $\lambda d_U(X,Y)+\lambda$ subsurfaces properly nested in $U$ such that if $V$ is properly nested in $U$  and $d_V(X,Y)>E$  then $V$ is nested in one of the surfaces in $Q$.  

\begin{proof}
The proof is via an iterative construction. The number of stages in the construction will be bounded by the maximum length of a properly nested chain of subsurfaces, so to prove uniform finiteness it suffices to prove that the number of subsurfaces added to $\cS$ in each stage is uniformly bounded. At the first stage we set $\cS=\{\Sigma\}$.

At each subsequent stage, we consider all the subsurfaces added to $\cS$ in the previous step that are not annular. Let $U$ be such a subsurface, and note that since $U$ is non-annular the map to $\cC(U)$ gives an un-parametrized quasi-geodesic. Let $H$ be such that for any Teichm\"uller geodesic and any non-annular curve graph $U$ there is a monotone reparametrization of the geodesic that stays within distance $H$ of a geodesic in $\cC(U)$. 

If it exists, let $X_U$ be the first point on the Teichm\"uller geodesic from $X$ to $Y$ with 
$$d_U(X, X_U)>M+H+10.$$
Similarly let $Y_U$ be the last point with 
$$d_U(Y_U, Y)>M+H+10.$$
If these points do not exist, we will work with the whole geodesic from $X$ to $Y$. If they do exist, we'll separately work with the segment from $X$ to $X_U$ and also the segment from $Y$ to $Y_U$. 

We now apply the large links axiom to these segments (of which there are one or two per $U$ added in the last stage) with subsurface $U$, and add all the resulting subsurfaces to $\cS$ and proceed to the next stage. 

The result is a uniformly finite collection $\cS$ of subsurfaces. Now, consider $U\in \cS$, and  $V$  properly nested in $U$ and assume
$$ d_V(X,Y)>C \quad\quad\text{and}\quad\quad d_U(\rho^V_U, X)<M.$$
We will specify $C$ soon. Note that the case $ d_U(\rho^V_U, Y)<M$ follows by symmetry from this case.

If $X_U$ is not defined above set $X_U=Y$. 
It suffices to prove that when $C$ large enough $$d_V(X, X_U)>E,$$ since this will give that in our construction a subsurface properly nested in $U$ and in which $V$ is nested is added to $\cS$ at the stage when the large links axiom is applied to the segment $[X, X_U]$ with subsurface $U$. 

To this end note that the geodesic from $\pi_U(X_U)$ to $\pi_U(Y)$ consists entirely of curves that cut $V$. Thus 
$$\rho^U_V(\pi_U(X_U), \pi_U(Y))\leq c_{BGI},$$
where $c_{BGI}$ is the constant in the Bounded Geodesic Image Theorem. Hence, keeping in mind that the constants required for the consistency conditions in this context are actually quite small, we get that 
$$d_V(X_U, Y) \leq c_{BGI}+100.$$
So, picking 
 $$C = E+c_{BGI}+200,$$
 ensures that 
 $d_V(X, X_V) >E$
as desired.
\end{proof}

\bold{The proof.} We can now prove the key proposition. 

\begin{proof}[Proof of Proposition \ref{P:key}.]
Let $K$ be a constant given by \cite[Theorem A]{LengthQuasiConvex}, so that the hyperbolic length of a curve at any point on a Teichm\"uller geodesic segment is at most $K$ times the maximum of its lengths at the endpoints of the segment. 

By Theorem \ref{T:Active} there is a constant $M_{\rm{active}}$ such that projection to a subsurface can change by at most  $M_{\rm{active}}$ during intervals where at least one boundary component has length at least $D/K$.  We can easily arrange for this also to apply to the modified annular projections of Appendix \ref{A:Rafi}, since the explicit estimates for the modified projections show in particular that progress only happens with the modulus is large, and large modulus implies the core curve is hyperbolically short. 
 
Let $\cS$ be as given in Lemma \ref{L:DM}, with any $M>\max(L, M_{active}).$ We can assume that the $C$ produced by that lemma is bigger than $L$, since the result with a bigger $C$ is weaker than the result with a smaller $C$. 

In general Lemma \ref{L:DM} gives for any surface $U$ not in $\cS$ where large progress occurs, there is a bigger subsurface in $\cS$ which proves that this progress does not happen in the initial or terminal part of the Teichm\"uller geodesic. By taking $M>L$, we ensure that that larger subsurface is never $\Sigma$ itself.  

Let $\alpha_1, \ldots, \alpha_m$ be the boundary curves of the subsurfaces in $\cS$ together with the set of all curves that are short at $X$ or $Y$. The number $m$ of such curves depends only on $M$ (and, as always, $g$ and $n$). 

Let $I_i$ be the interval from the first time $\alpha_i$ has length less than $D/K$ to the last time it has length $D/K$. The main result of \cite{LengthQuasiConvex} gives that $\alpha_i$ has length at most $D$ at every point of $I_i$.

Let $X', Y'$ be in the same component of $[X,Y]$ minus the segments corresponding to the $I_i$. By Proposition \ref{P:KasraUnique}, to conclude this proof it suffices to show that $X'$ and $Y'$ are bounded distance in all non-annular curve graphs and in the modified annular projections of all curves. We will do this by contradiction. 

Suppose that $U$ is a non-annular subsurface with $$d_U(X', Y')>C+M_{\rm{active}}+100.$$ We immediately get that $U$ cannot be in $\cS$. It follows that $U$ is properly nested in a proper subsurface $V$ in $\cS$, and that the progress in $U$ is not initial or terminal as measured in $\cC(V)$. It follows that the active interval for $U$ is contained in the active interval for $V$, giving a contradiction to the fact that no $\alpha_i$ is short between $X'$ and $Y'$. 

Now suppose that $X'$ and $Y'$ are far apart as measured using the modified annular projection of Appendix \ref{A:Rafi} for some curve $\alpha$. It follows from Corollary \ref{C:NewNoBacktracking} that $X$ and $Y$ are far apart using the same modified annular projection. Rafi has proven that when $\alpha$ is not short, the regular and modified annular projections agree \cite[Theorem 4.3]{RafiComb}, so we either get that $X$ and $Y$ are far apart using the usual annular projections or we get that $\alpha$ is short at $X$ or $Y$. 

\bold{Case 1:}  $\alpha$ is short at $X$ or $Y$.  Then $\alpha=\alpha_i$ for some $i$, and we get that $\alpha$ is not short along $[X', Y']$. This contradicts the fact that the geodesic makes progress in the modified annular projection during this time.

\bold{Case 2:} $X$ and $Y$ are far apart using the usual annular projections and $\alpha$ is not short at $X$ or $Y$. 
Then we know from Lemma \ref{L:DM} that $\alpha$ must be nested in a proper subsurface $V$ which is contained in $\cS$, and that in $\cC(V)$ we have that $\alpha$ lies far from $X$ and far from $Y$. The progress of the modified projection thus happens far from $X$ or $Y$ in $\cC(V)$, showing that the interval when the modified projection changes is contained in the active interval for $V$, again giving a contradiction. 
\end{proof}

\bibliographystyle{amsalpha}

\bibliography{bibliography}

\end{document}